\documentclass[a4paper,12pt]{article}
\usepackage[utf8]{inputenc}

\makeatletter
\newif\if@check@engine  \@check@enginetrue 
\makeatother
 \usepackage[pdf]{pstricks}
\newcommand{\skipfig}[1]{}

\usepackage{amsmath}
\usepackage{amsthm}
\usepackage{amssymb}

\usepackage{latexsym}
\usepackage{array}
\usepackage{enumerate, comment}
\usepackage[justification=centering]{caption}
\usepackage{abstract}
\usepackage{tikz-cd}

\tikzset{
  symbol/.style={
    draw=none,
    every to/.append style={
      edge node={node [sloped, allow upside down, auto=false]{$#1$}}}
  }
}

\newcommand{\CC}{(\mathbb{C}\setminus 0)}
\newcommand{\C}{\mathbb{C}}

\newcommand{\R}{\mathbb{R}}
\newcommand{\Q}{\mathbb{Q}}
\newcommand{\Z}{\mathbb{Z}}

\newcommand{\CP}{\mathbb{CP}}

\renewcommand{\dim}{{\rm dim}\,}
\newcommand{\codim}{{\rm codim}\,}
\newcommand{\Vol}{{\rm Vol}\,}
\newcommand{\MV}{{\rm MV}\,}

\newcommand{\Trop}{{\rm Trop}\,}

\newcommand{\supp}{{\rm supp}\,}

\newtheorem{theorem}{Theorem}[section]
\newtheorem{proposition}[theorem]{Proposition}
\newtheorem{lemma}[theorem]{Lemma}
\newtheorem{corollary}[theorem]{Corollary}
\theoremstyle{definition}
\newtheorem{definition}[theorem]{Definition}
\newtheorem{example}[theorem]{Example}
\newtheorem{remark}[theorem]{Remark}
\newtheorem{conjecture}[theorem]{Conjecture}
\newtheorem*{conv}{Convention}
\newtheorem{problem}[theorem]{Problem}

\begin{document}

\title{The ring of local tropical fans \\ and tropical nearby monodromy eigenvalues} 
\author{Alexander ESTEROV 
\thanks{HSE University \newline
{\bf E-mail:} aesterov@hse.ru \newline
The author is partially supported by International Laboratory of Cluster Geometry NRU HSE, RF Government grant, ag. 075-15-2021-608 from 08.06.2021.
}}

\maketitle

\begin{abstract} We extend the tropical intersection theory to tropicalizations of germs of analytic sets. In particular, we construct a (not entirely obvious) local version of the ring of tropical fans with a nondegenerate intersection pairing.  

As an application, we study nearby monodromy eigenvalues -- the eigenvalues of the monodromy operators of singularities, adjacent to a given singularity of a holomorphic function $f$.  
More precisely, we express some of such values in terms of certain resolutions of $f$.
The expression is given it terms of the exceptional divisor strata of arbitrary codimension, generalizing the classical A'Campo formula that consumes only codimension 1 strata and produces only monodromy eigenvalues at the origin. For this purpose, we introduce tropical characteristic classes of germs of analytic sets, and use this calculus to detect some of the nearby monodromy eigenvalues, which we call tropical. 

The study is motivated by the monodromy conjectures by Igusa, Denef and Loeser: every pole of an appropriate local zeta function of $f$ induces a nearby monodromy eigenvalue. We propose a presumably stronger version of this conjecture: all poles of the local zeta function induce tropical nearby monodromy eigenvalues.

In particular, if the singularity is non-degenerate with respect to its Newton polyhedron $N$, then the tropical monodromy eigenvalues can be expressed in terms of fiber polytopes of certain faces of $N$, so our conjecture (unlike the original ones) becomes a purely combinatorial statement about a polyhedron. This statement is confirmed for the topological zeta function in dimension up to 4 in a joint work with A. Lemahieu and K. Takeuchi, which, in particular, supports our conjecture and proves the original one for non-degenerate singularities in 4 variables.
\end{abstract}

\section{Introduction}

\subsection{Local tropical intersection theory}

To every (very) affine algebraic variety, one can associate its tropical fan, so that:

i) tropical fans form a graded ring $\bigoplus K_k$ (a combinatorial model for the ring of conditions of the complex torus);

ii) taking the product of fans of complementary dimension in this ring defines their non-degenerate intersection pairing, taking values in the 0-dimensional component $K_0=\R$;

iii) the intersection number of varieties in general position equals the intersection number of their tropical fans.

When trying to extend this tropical intersection theory from affine algebraic sets to germs of analytic sets in $(\C^n,0)$, there is a natural way to define the ``local tropical fan'' of an analytic germ, so that, for a germ of an algebraic set $V\subset\C^n$ at 0, its local fan is just the intersection of the conventional global topical fan of $V$ with the (open) positive quadrant (see e.g. \cite{local1}, \cite{local2} and \cite{local3}).

Such local tropical fans do form a ring, but the intersection pairing in this ring is trivial, as its 0-dimensional component is.

In this paper, we embed this ring into a certain larger ring of so called relative tropical fans (Definition \ref{defreltrop}), which does have a nondegenerate intersection pairing, and still has algebraic geometry meaning, see Section 4.

This allows us to define the tropical characteristic class of a germ of analytic set (the limit of the CSM classes in arbitrarily rich toric resolutions, Definition \ref{deftcc} below), similarly to the case of algebraic sets \cite{E17}, \cite{gross}, \cite{shaw} (the limit of the CSM classes in arbitrarily rich toric compactifications), and apply them to the study of monodromy of complex analytic functions, with a view towards monodromy conjectures.

\subsection{Polyhedral monodromy conjecture}

Let $f:(\C^n,0)\to(\C,0)$ be a germ of a holomorphic function. For all sufficiently small $x_0\in\C^n$, the {\it nearby singularity} germ $$f_{x_0}:(\C^n,0)\to(\C,0),\quad f_{x_0}(x)=f(x_0+x),
$$ is defined. We shall define  {\it nearby monodromy eigenvalues} of $f$ as the complex numbers that appear as eigenvalues of the Milnor monodromy of the nearby singularities at arbitrarily small $x_0$.

The aim of this paper is to study these numbers in terms of a given resolution of singularities of $f$ at $0$, and, particularly, in terms of the Newton polyhedron of the germ $f$ at $0$ provided it is non-degenerate.

The latter question might seem to be trivial at the first glance: every nearby monodromy eigenvalue of $f$ is a root or pole of the Milnor monodromy $\zeta$-function $\zeta_{x_0}$ of the germ $f_{x_0}$ at a suitable point $x_0$, and the Varchenko formula \cite{V} expresses $\zeta_{x_0}$ in terms of the Newton polyhedron of $f_{x_0}$.

However, this recipe will not work in high dimension (for $n\geqslant 4$): the germ $f_{x_0}$ may be degenerate with respect to its Newton polyhedron even for non-degenerate $f$, so $\zeta_{x_0}$ may be out of reach to the Varchenko formula: see Example 7.4 in \cite{ELT}, which inspired this work.

As a possible solution, we shall assign to every lattice polyhedron $N$ the set of its {\it tropical nearby monodromy eigenvalues} (abbreviating to t.n.e., see Definition \ref{deftropicalpoly} below) with the following property. 
\begin{theorem}[Corollary \ref{tropicalpoly}]\label{thmain1} Every t.n.e. of the polyhedron $N$ is a nearby monodromy eigenvalue of every non-degenerate singularity $f$ with the Newton polyhedron $N$.
\end{theorem}
\begin{remark} 1. T.n.e.'s include all roots and poles of the monodromy $\zeta$-function $\zeta_{0}$, but in general we do not expect them to contain all nearby monodromy eigenvalues of $f$ (neither do we have examples of such $f$).

2. The reader interested in the definition of t.n.e.'s of polyhedra, can directly proceed from Section \ref{sspecnewton0} (where the dual tropical fans of polyhedra are introduced) to Section \ref{sspecnewton} (where we give the definition), and further to Section 6 (where we simplify it in terms of fiber polyhedra).

The part of the paper between Sections \ref{sspecnewton0} and \ref{sspecnewton} is devoted to establishing a technique to prove that t.n.e.'s of the Newton polyhedron of $f$ are really nearby monodromy eigenvalues of $f$.
\end{remark}

The study of nearby monodromy eigenvalues is motivated in particular by the Denef--Loeser monodromy conjecture \cite{DL}: for every pole $q$ of the topological $\zeta$-function of $f$, the number $\exp{2\pi i q}$ is a nearby monodromy eigenvalue. We expect that our notion of t.n.e. allows to make this conjecture more constructive as follows. Recall that Denef and Loeser \cite{DL} found an expression for the topological $\zeta$-function of a non-degenerate singularity $f$ in terms of its Newton polyhedron $N$; we denote this expression by $Z_N$. 
\begin{conjecture}[Polyhedral Denef--Loeser conjecture]\label{conjpoly} For every polyhedron $N$ and every pole $q$ of $Z_N$, the number $\exp{2\pi i q}$ is a t.n.e. of $N$.
\end{conjecture}
Unlike the original Denef--Loeser conjecture for non-degenerate singularities, this (presumably stronger) conjecture is a purely combinatorial statement about lattice polyhedra in all dimensions. For $n\leqslant 4$, it is proved in \cite{ELT} (which, together with Theorem \ref{thmain1}, implies the original Denef--Loeser conjecture for non-degenerate singularities of functions of four variables).  
\subsection{Tropical nearby monodromy eigenvalues via embedded toric resolutions}
The scope of this paper is not restricted to singularities that are non-degenerate with respect to Newton polyhedra: we shall actually introduce the notion of tropical nearby monodromy eigenvalues in the more general setting of embedded toric resolutions in the sense of Teissier \cite{teissier}. More specifically, let $f:(\mathcal{S},0)\to(\C,0)$ be a holomorphic function on a germ of an analytic set. We do not impose any non-degeneracy assumptions this time. The case $(\mathcal{S},0)=(\C^n,0)$ is relevant to the original Denef--Loeser conjecture, but the computations of t.n.e.'s below will be valid for any $\mathcal{S}$.
\begin{definition}\label{defemb} An embedded toric resolution of $f$ consists of a germ of a proper embedding $i:(\mathcal{S},0)\hookrightarrow(\C^N,0)$ and a toric blow-up $\pi:Z\to\C^N$ with the set of orbits $\mathcal{E}$, satisfying the following properties:

1. The preimage $\pi^{-1}(0)$ is a compact union of a certain collection of orbits $\mathcal{E}_\varnothing\subset\mathcal{E}$.

2. The images $i(\mathcal{S})$ and $i(f=0)$ intersect the torus $\CC^N$ by their dense subsets, denoted by $S$ and $X$ respectively.

3. The strict transforms of $X\subset S$ under $\pi$, denoted by $Y\subset W$, are smooth in a neighborhood of $\pi^{-1}(0)$ and intersect transversally each of the orbits $E\in\mathcal{E}_\varnothing$ (and hence each of $E\in\mathcal{E}$) by smooth sets $Y_E$ and $W_E$ respectively.
\end{definition}

For every germ of a holomorphic function $f$, admitting an embedded toric resolution $\pi$, we shall define the set of {\it tropical nearby monodromy eigenvalues (t.~n.~e.)} $M_{f,\pi}$ (see Definition \ref{deftne} below) that depends only on the combinatorics of $Z$ and the Euler characteristics of the connected components of the strata $W_E\setminus Y_E,\, E\in\mathcal{E}_\varnothing$. 
\begin{theorem}[Theorem \ref{main} below]\label{th15} For every embedded toric resolution $\pi$ of a singularity $f$, every t.~n.~e. is indeed a nearby monodromy eigenvalue.
\end{theorem}
\begin{conjecture}[Constructive Denef--Loeser conjecture]\label{conjconstr} For every embedded toric resolution $\pi$ of a singularity $f:(\C^n,0)\to(\C,0)$ and every pole $q$ of the topological $\zeta$-function of $f$, we have $\exp{2\pi i q}\in M_{f,\pi}$.
\end{conjecture}

The proof of the theorem follows from the properties of tropical characteristic classes of analytic germs. Namely, the Euler characteristics of the strata $Y_E$ and $W_E$ for $E\in\mathcal{E}_\varnothing$ are related to nearby monodromy eigenvalues of $f$ via the tropical characteristic classes of the strata $Y_E$ and $W_E$ for $E\in\mathcal{E}$ as follows.

-- On one hand, using the adjunction property of tropical characteristic classes (Corollary \ref{schon}), we can express the tropical characteristic classes of the strata $Y_E$ and $W_E$ for $E\in\mathcal{E}$ in terms of the Euler characteristics of the strata $Y_E$ and $W_E$ for $E\in\mathcal{E}_\varnothing$.

-- On the other hand, applying the pushforward property of tropical characteristic classes (Definition \ref{deftcc}.4) to the projection $\pi$, we extract from the tropical characteristic classes of $Y_E$ and $W_E$ for $E\in\mathcal{E}$ the Euler characteristics of the fibers of these strata under the projection $\pi$ over a nearby point $x\in S$, and then, via the A'Campo formula, some monodromy eigenvalues at $x$.

\begin{remark} 1. We extract nearby monodromy eigenvalues from the topology of strata $E\in\mathcal{E}_\varnothing$ of arbitrary codimension, while the Varchenko \cite{V} and, more generally, A'Campo \cite{AC} formulas consume only codimension 1 strata of the exceptional divisor and produce only monodromy eigenvalues at the origin. Roughly speaking, we shall say that a t.~n.~e. has codimension $d$, if it is extracted from orbits of codimension at most $d+1$ (see Definition \ref{deftne} for details).

2. An embedded toric resolution is expected to exist for every singularity $f:(\C^n,0)\to(\C,0)$. Tevelev \cite{tev2} proved the existence of embedded toric resolutions in a somewhat weaker sense (omitting the assumption $i(0)=0$ of Definition \ref{defemb}, because this condition makes no sense in the global setting studied by Tevelev). Since the present work does not readily extend to the setting of arbitrary Hironaka resolutions, it would be interesting to extend Tevelev's theorem to our local setting.

3. Conjecture \ref{conjpoly} is a special case of Conjecture \ref{conjconstr} for non-degenerate singularities. However, restricting our attention only to non-degenerate singularities and their classical toric resolutions would not simplify the paper. In this sense embedded toric resolutions provide a natural setting for the study of t.~n.~e.

4. The first essential open question about t.~n.~e. is: whether the set $M_{f,\pi}$ coincides with the set of {\it all} nearby monodromy eigenvalues of $f$? More specifically: 

-- Is this true for a suitable resolution $\pi$ of every singularity $f$ (so that Theorem \ref{th15} is capable of computing all nearby monodromy eigenvalues)?

-- Is this wrong for {\it some} resolutions $\pi$ of sufficiently complicated singularities $f$ (making Conjecture \ref{conjconstr} strictly stronger than the original Denef--Loeser conjecture)?

The answer to the latter question can be (and expected to be) positive only starting from dimension $n=5$, so the question is non-trivial.


5. Despite the expectation that Conjectures \ref{conjpoly} and \ref{conjconstr} may be stronger than the original Denef--Loeser conjecture, in the non-degenerate case they happen to be simpler to prove than the original one.

6. It would be natural to expect that Conjectures \ref{conjpoly} and \ref{conjconstr} extend to the poles of local Igusa's p-adic zeta functions (see \cite{Ig}): for large $p$ and $k$, the poles of $Z_{f,\Phi}(s)=\int_{\Q^n_p} |f(x)|^s\Phi(x)|dx|$ with $\Phi=($the characteristic function of $p^k\Z^n_p)$ give rise to t.~n.~e. of $f$. The first opportunity to support this conjecture will be to prove it for non-degenerate singularities of four variables, extending \cite{ELT} to four variables in the same way as \cite{B-V} extends \cite{L-V} for three variables.

Yet more optimistically, one could expect that t.~n.~e. are enough for the motivic monodromy conjecture \cite{dl98}. 

7. In general, this paper is intended to be a step towards a tropical approach to the motivic monodromy conjecture. Another way to apply tropical geometry to the study of monodromy was presented earlier in a series of papers \cite{ks}, \cite{ks2} and \cite{s}, where the motivic Milnor fiber is studied in terms of tropical degenerations. In particular, this leads to a different refinement of the A'campo and Varchenko formulas for monodromy eigenvalues (describing the Jordan structure of the monodromy operator at the origin). However, the authors of \cite{ELT} were unable to extract the missing nearby monodromy eiganvalues from these results, so the present paper seems to be in a sense complementary to them.
\end{remark}

\subsection{Structure of the paper}

The definition of tropical nearby monodromy eigenvalues relies upon the notion of tropical fans and tropical characteristic classes of germs of complex analytic sets, which we introduce in Sections 2 and 3. Originally, tropical characteristic classes were introduced in \cite{E17} for algebraic sets with a view towards applications in enumerative geometry.
In Section 4 we adapt the proof of the existence of tropical characteristic classes to the case of analytic germs. 

In Section 5 we introduce tropical nearby monodromy eigenvalues and prove that they are indeed nearby monodromy eigenvalues. Then in Section \ref{sspecnewton} we specialize to the case of non-degenerate singularities. In Section 6 we use the calculus of fiber polyhedra to make the definition of tropical nearby monodromy eigenvalues of small codimension more explicit -- the result of this simplification (Corollary \ref{coroldim4}) is used in \cite{ELT} to prove Conjecture \ref{conjpoly} in dimension 4.

Extending the results of Section 6 to higher codimension boils down to certain positivity questions for fiber polyhedra, resembling dual defectivity (see Problem \ref{probldef} and subsequent remarks).

\begin{remark}
1. The reader that prefers to take for granted the existence of tropical characteristic classes of analytic sets may skip Sections \ref{ssmulttrop} and 4.

2. The reader interested only in the definition of t.n.e.'s in the setting of non-degenerate functions and their Newton polyhedra, can directly proceed from Section \ref{sspecnewton0} to Section \ref{sspecnewton}.
\end{remark}

\section{Local tropical fans}\label{schoice1}

\begin{conv}
There are two common versions of the definition of support faces and support functions: 

$\bullet$ A face $F$ of a polyhedron $N$ in a real vector space $V$ and a linear function $v\in V^*$ are said to support each other, if the restriction of $v$ to $N$ takes its {\it maximal} (or {\it minimal}) values at the points of $F$ (in this case, $F$ is denoted by $N^v$). 

$\bullet$ The support function of $N$ is the function $N(\cdot):V^*\to\R$, whose value at a linear function $v\in V^*$ is the {\it maximum} (or {\it minimum}) of $v$ restricted to $N$. 

$\bullet$ The dual cone $C^*\subset V^*$ to a given open cone $C\subset V$ is the set of all support vectors of the vertex $0\in C$, i.e. the ones taking {\it negative} (or {\it positive}) values on $C$ respectively.\end{conv}

The first convention agrees with the $(\max,+)$ notation in tropical geometry and is typical in the study of convex bodies and projective toric varieties. It however implies that the dual cone to the positive orthant in $\R^n$ is the negative orthant, so this object is sometimes called the polar cone to distinguish it from the conventional one. Due to this issue, the second convention is more common in the study of toric resolutions of singularities.

Since the present paper is at the border of the two fields, we never take a side: all the computations below are written in a way intended to be independent of this choice. The only exception is Section 4, where we choose the first (maximal/negative) convention to work with (quasi-)projective toric varieties.
\begin{definition}\label{defambient} An $n$-dimensional {\it ambient cone} $\mathcal{C}$ is a pair $(L,C)$, where $L\cong \Z^n$ is a lattice, and $C$ is an open convex rational polyhedral cone in the vector space $L\otimes\Q$. In the special case $C=L\otimes\Q$, the ambient cone will be denoted by the same letter $L$. If $C$ contains the kernel of a given lattice homomorphism $p$, then the image $p(\mathcal{C})$ is defined as the ambient cone $(p(L),p(C))$, otherwise $p(\mathcal{C})$ is not defined.\end{definition}

Now the readers familiar with the ring of tropical fans (see e.g. \cite{fs}, \cite{kaz}, \cite{sty}, 
\cite{tropfans}, \cite{mcst}) can proceed directly to Definition \ref{defloctrop} at the end of this section. For other readers, we 
introduce from scratch the notion of weighted and tropical fans in the context of ambient cones (along with some notation that we shall use later), and refer to the aforementioned papers for the proofs.

\subsection{Weighted fans} A $k$-dimensional {\it weighted pre-fan} $F$ in the ambient cone $\mathcal{C}$ is a finite union $S_F\subset C$ of disjoint 
$k$-dimensional relatively open convex rational polyhedral cones, endowed with a locally constant function $w_F:S_F\to\Q$. The set $S_F$ and the function $w_F$ are called the {\it support set} and the {\it weight function} of $F$.

A {\it weighted fan} is an equivalence class of pre-fans with respect to the minimal equivalence relation generated by the following two instances:

(1) pre-fans $F$ and $G$ are equivalent, if $S_F$ is an (open) dense subset of $S_G$, on which $w_F=w_G$.

(2) pre-fans $F$ and $G$ are equivalent, if the weights $w_F$ and $w_G$ are equal on the intersection $S_F\cap S_G$ and vanish outside of it.

\begin{example}
Informally speaking, this equivalence relation allows to change  the support set of the pre-fan by adding/removing $k$-dimensional piecewise-linear sets with zero weights or less than $k$-dimensional piecewise-linear sets with arbitrary weights. For instance, gray sets and numbers on Figure 1 represent support sets and weight functions of several equivalent pre-fans in the plane.
\end{example}
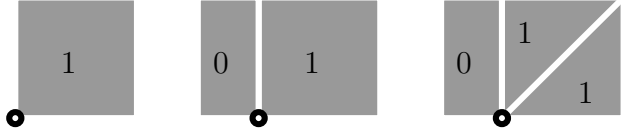
\begin{figure}
\centering\includegraphics{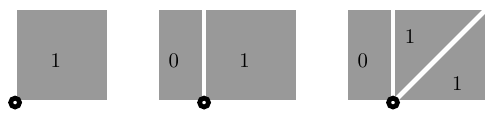}\skipfig{
\psscalebox{1.0 1.0} 
{
\begin{pspicture}(0,-0.88)(8.16,0.88)
\definecolor{colour0}{rgb}{0.6,0.6,0.6}
\psline[linecolor=white, linewidth=0.08, fillstyle=solid,fillcolor=colour0](0.12,0.84)(0.12,-0.76)(1.72,-0.76)(1.72,0.84)(0.12,0.84)
\psline[linecolor=white, linewidth=0.08, fillstyle=solid,fillcolor=colour0](3.32,0.84)(3.32,-0.76)(4.92,-0.76)(4.92,0.84)(3.32,0.84)
\psline[linecolor=white, linewidth=0.08, fillstyle=solid,fillcolor=colour0](2.52,0.84)(2.52,-0.76)(3.32,-0.76)(3.32,0.84)(2.52,0.84)
\psline[linecolor=white, linewidth=0.08, fillstyle=solid,fillcolor=colour0](6.52,0.84)(6.52,-0.76)(8.12,-0.76)(8.12,0.84)(6.52,0.84)
\psline[linecolor=white, linewidth=0.08, fillstyle=solid,fillcolor=colour0](5.72,0.84)(5.72,-0.76)(6.52,-0.76)(6.52,0.84)(5.72,0.84)
\psline[linecolor=white, linewidth=0.08](6.52,-0.76)(8.12,0.84)
\pscircle[linecolor=black, linewidth=0.08, dimen=inner](0.12,-0.76){0.04}
\pscircle[linecolor=black, linewidth=0.08, dimen=inner](6.52,-0.76){0.04}
\pscircle[linecolor=black, linewidth=0.08, dimen=inner](3.32,-0.76){0.04}
\rput[bl](0.72,-0.16){$1$}
\rput[bl](3.92,-0.16){$1$}
\rput[bl](6.72,0.24){$1$}
\rput[bl](7.52,-0.56){$1$}
\rput[bl](2.72,-0.16){$0$}
\rput[bl](5.92,-0.16){$0$}
\end{pspicture}
}}
\caption{several equivalent two-dimensional pre-fans in the plane}
\end{figure}

If a weighted fan $\mathcal{F}$ is represented by a pre-fan $F$, we shall sometimes write 
$w_\mathcal{F}$ instead of 
$w_F$, if the choice of the representative does not matter. The support set of a weighted fan $\mathcal{F}$, denoted by $S_\mathcal{F}$, or $\supp\mathcal{F}$, is the closure of the set defined by the inequality $w_\mathcal{F}\ne 0$.

\subsection{Operations with weighted fans} 
To define the product $q\cdot\mathcal{F}$ for a weighted fan $\mathcal{F}$ and a number $q\in\Q$, multiply the weight function of an arbitrary representative of $\mathcal{F}$ by $q$.
To define the sum of two weighted fans, choose their representatives $F$ and $G$ such that $S_F=S_G$, and define the sum as the equivalence class of the pre-fan $(S_F, w_F+w_G)$.

Equipped with these two operations, the set of all weighted fans in the $n$-dimensional ambient cone $\mathcal{C}=(L,C)$ is a vector space denoted by $Z_k(\mathcal{C})$. In particular, every $k$-dimensional rational polyhedral cone $P\subset C$, endowed with the unit weight, gives rise to a fan that will be denoted by the same letter $P\in Z_k(\mathcal{C})$. These fans (over all cones $P$) generate $Z_k(\mathcal{C})$. 

\begin{definition}\label{defprojfan}
An epimorphism of ambient cones $p(\mathcal{C})=\mathcal{C}'$ gives rise to the direct image linear map $p_*:Z_k(\mathcal{C})\to Z_k(\mathcal{C}')$, defined on every convex cone $P\subset C$ as follows:

-- if $\dim p(P)<k$, then we set $p_*(P)=0$.

-- if $\dim p(P)=k$, then, denoting the vector span of $P$ by $U$ and the lattice index $$\left|\frac{p(U)\cap p(L)}{p(U\cap L)}\right|$$ by $d$, we set $p_*(P)=d\cdot p(P)$.\end{definition}

\begin{example} See Figure \ref{picsum} for an example of the sum and the direct image of one-dimensional weighted fans in the plane.
\end{example}
\begin{figure}
\centering\includegraphics{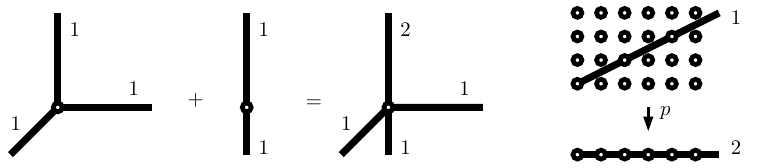}
\caption{the sum and the direct image of one-dimensional fans}\label{picsum}
\end{figure}

The localization $P_v$ of a cone $P$ at $v\in L$ is the unique cone that is invariant under the shift by $v$ and coincides with $P$ in a small neighborhood of $v$ (equivalently, $P_v$ is generated by $P$ and $-v$ if $v\in\bar P$, and $P_v=\varnothing$ otherwise). This operation extends by linearity to the localization $\mathcal{F}_v\in Z_k(\mathcal{C}_v)$ of a weighted fan $\mathcal{F}\in Z_k(\mathcal{C})$, where $\mathcal{C}_v=(L,C_v)$ is the localization of the ambient cone $\mathcal{C}=(L,C)$.

\begin{example} 
The second summand on Figure \ref{picsum} is the localization of the first one at $v=(0,1)$, while the localization at $v=(0,-1)$ is empty. \end{example}

\subsection{Dual fans of polyhedra}\label{sspecnewton0} An important example of tropical fans and their products comes from dual fans of lattice polyhedra. Recall that the Minkowski sum $P+Q$ for convex sets $P$ and $Q$ is defined as $\{p+q\,|\, p\in P,\, q\in Q\}$, and consider a lattice polyhedron $N\subset L^*\otimes\Q$ of the form $C^*+($bounded set$)$ in the dual space to the ambient cone $\mathcal{C}=(L,C)$. 
\begin{definition}\label{defdual} The $k$-th dual fan $[N]^k\in K^k(\mathcal{C})$ is defined by a pre-fan $F$ such that $S_F$ is the set of all vectors $v\in L\otimes\Q$, whose support face $N^v\subset N$ is bounded and $k$-dimensional, and $w_F(v)$ is the $k$-dimensional lattice volume of the face $N^v$. \end{definition} 

Recall that the lattice volume is the volume form on a $k$-dimensional vector subspace $U\subset\Q^n$ such that the volume of $U/(U\cap\Z^n)$ equals $k!$; the latter factor assures that the lattice volume of every lattice polytope is an integer.  
See Figure \ref{picdual} for an example of a dual fan.

\begin{figure}\centering\includegraphics{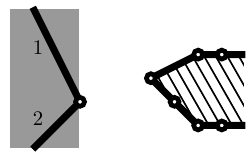}\skipfig{
\psscalebox{1.0 1.0} 
{
\begin{pspicture}(0,-1.2346193)(4.04,1.2346193)
\definecolor{colour0}{rgb}{0.6,0.6,0.6}
\psline[linecolor=white, linewidth=0.04, fillstyle=solid,fillcolor=colour0](1.22,1.2078072)(1.22,-1.1921928)(0.02,-1.1921928)(0.02,1.2078072)(1.22,1.2078072)
\psline[linecolor=black, linewidth=0.12](0.42,1.2078072)(1.22,-0.39219284)
\psline[linecolor=black, linewidth=0.12](0.42,-1.1921928)(1.22,-0.39219284)
\pscircle[linecolor=black, linewidth=0.08, fillstyle=solid, dimen=inner](1.22,-0.39219284){0.04}
\psline[linecolor=white, linewidth=0.04, fillstyle=vlines, hatchwidth=0.028222222, hatchangle=30.0, hatchsep=0.1411111](4.02,0.40780717)(3.22,0.40780717)(2.42,0.0078071593)(3.22,-0.7921928)(4.02,-0.7921928)(4.02,0.40780717)
\psline[linecolor=black, linewidth=0.12](4.02,0.40780717)(3.22,0.40780717)(2.42,0.0078071593)(3.22,-0.7921928)(4.02,-0.7921928)
\pscircle[linecolor=black, linewidth=0.08, fillstyle=solid, dimen=inner](3.22,0.40780717){0.04}
\pscircle[linecolor=black, linewidth=0.08, fillstyle=solid, dimen=inner](2.82,-0.39219284){0.04}
\pscircle[linecolor=black, linewidth=0.08, fillstyle=solid, dimen=inner](2.42,0.0078071593){0.04}
\pscircle[linecolor=black, linewidth=0.08, fillstyle=solid, dimen=inner](3.62,-0.7921928){0.04}
\pscircle[linecolor=black, linewidth=0.08, fillstyle=solid, dimen=inner](3.62,0.40780717){0.04}
\pscircle[linecolor=black, linewidth=0.08, fillstyle=solid, dimen=inner](3.22,-0.7921928){0.04}
\rput[bl](0.42,0.40780717){$1$}
\rput[bl](0.42,-0.7921928){$2$}
\end{pspicture}
}}
\caption{the gray cone $C$, the hatched polyhedron $N=C^*+($a bounded set$)$, and its dual fan $[N]$}\label{picdual}
\end{figure}

\begin{remark}\label{remembed} 
1. In the next subsection we introduce a subclass of weighted fans called tropical fans and define their multiplication, so that the fan $[N]^k$ is tropical and equals the product of $k$ copies of $[N]^1=:[N]$ (hence the notation).
In particular, if the polyhedron $N$ is bounded, then $[N]^n\in K_0(L)=\Q$ equals the lattice volume of $N$. More generally, the product of fans $[N_1],\ldots,[N_n]$ is the mixed volume of the polytopes $N_1,\ldots,N_n$.  

2. Note that $[N_1+N_2]=[N_1]+[N_2]$, and a polyhedron $N$ is uniquely determined by its dual fan $[N]$, so the homomorphism map $N\mapsto[N]$ embeds the semigroup of polyhedra into the group of codimension 1 tropical fans $K^1(\mathcal{C})$. Moreover, the latter is the Grothendieck group of the former.

3. At this point the reader interested in the definition of t.n.e.'s of polyhedra may proceed to Section\ref{sspecnewton}.
\end{remark}

\subsection{Ring of tropical fans}\label{ssmulttrop} We introduce subspaces of {\it local tropical fans} $K_k(\mathcal{C})\subset Z_k(\mathcal{C})$ such that the sum $\bigoplus_{k=0}^n K_k(\mathcal{C})=K(\mathcal{C})$ has a natural ring structure. 
\begin{remark} In what follows, the notion of tropical fans and their ring structure is essentially used only in Section 4 to prove the existence of tropical characteristic classes.
\end{remark}

Let us start with the classical special case $\mathcal{C}=(\Z^n,\Q^n)$.
A $k$-dimensional fan $\mathcal{F}\in Z_{k}(\Z^n,\Q^n)$ is said to be tropical, if, for every epimorphism $p:\Z^n\to\Z^k$, the fan $p(\mathcal{F})\in Z_{k}(\Z^k,\Q^k)$ is a multiple of $\Q^k$. The space of all $k$-dimensional tropical fans is denoted by $$K_{k}(\Z^n,\Q^n)=K^{n-k}(\Z^n,\Q^n).$$ If $\mathcal{F}_i\in K^{k_i}(\Z^n,\Q^n)$ are tropical fans such that $\sum_{i=1}^N k_i=n$, then their Descartes product $$\mathcal{F}_1\times\ldots\times\mathcal{F}_N\in Z^{n}(\Z^{n}\times\ldots\times\Z^{n},\Q^{n}\times\ldots\times\Q^{n})$$ is also a tropical fan, so its direct image under the projection along the diagonal $$\Q^{Nn}\to\Q^{(N-1)n}$$ equals $d\cdot\Q^{(N-1)n}$. This number $d\in\Q$ is called the intersection number of $\mathcal{F}_i$ and is denoted by $\mathcal{F}_1\circ\ldots\circ\mathcal{F}_k$. This pairing in the space $$K(\Z^n,\Q^n)=\bigoplus_{k=0}^n  K^{k}(\Z^n,\Q^n)$$ 
extends to the unique commutative associative multiplication $$\cdot\;:\;K^{k}(\Z^n,\Q^n)\times K^{m}(\Z^n,\Q^n)\to K^{k+m}(\Z^n,\Q^n)$$ such that $$(\mathcal{F}\cdot\mathcal{G})\circ\mathcal{H}=\mathcal{F}\circ\mathcal{G}\circ\mathcal{H}$$
for all $\mathcal{F}, \mathcal{G}$ and $\mathcal{H}$.

We now exted to our context a more constructive way \cite{fs} to define tropical fans and their products, which is more convenient when it comes to explicit computations. 

\begin{definition}[Tropicality via balancing condition]
I. Let $\mathcal{F}\in Z_1(\Z^n,\Q^n)$ be a 1-dimensional weighted fan, i.e. $\supp\mathcal{F}$ consists of finitely many rays, generated by primitive lattice vectors $v_1,\ldots,v_N$, and the weight function $w_\mathcal{F}$ takes the value $w_i$ at $v_i$. Then $\mathcal{F}$ is said to be tropical, if it satisfies the balancing condition $$\sum_i w_i v_i=0.$$

II. A weighted fan $\mathcal{F}\in Z_k(\Z^n,\Q^n)$ is said to be a {\it book}, if $\Z^n$ 
admits a decomposition $L_1\oplus L_2$ such that $\mathcal{F}=\mathcal{F}_1\times\mathcal{F}_2$, where $\mathcal{F}_1\in Z_1(L_1)$ is 1-dimensional, and $\mathcal{F}_2\in Z_{k-1}(L_2)$ is the whole space $L_2\otimes\Q$ (see Figure \ref{picbook}). Then $\mathcal{F}$ is said to be tropical, if $\mathcal{F}_1$ is.

III. An arbitrary weighted fan $\mathcal{F}\in Z_k(\mathcal{C})$ is said to be tropical, if, for every point $v\in C$, such that the localization $\mathcal{F}_v\in Z_k(\Z^n,\Q^n)$ is a book, this book is tropical. Tropical fans form a vector subspace $K_k(\mathcal{C})\subset Z_k(\mathcal{C})$. 
\end{definition}

\begin{figure}
\centering\includegraphics{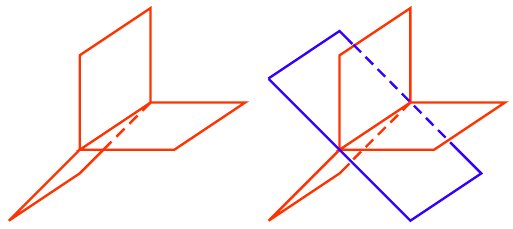}\skipfig{
\psscalebox{1.0 1.0} 
{
\begin{pspicture}(0,-1.8315656)(8.48044,1.8315656)
\definecolor{colour0}{rgb}{1.0,0.2,0.0}
\definecolor{colour1}{rgb}{0.2,0.0,1.0}
\psline[linecolor=colour0, linewidth=0.04](1.2141422,-0.605821)(1.2141422,0.994179)(2.4141421,1.794179)(2.4141421,0.194179)(1.2141422,-0.605821)(2.8141422,-0.605821)(4.014142,0.194179)(2.4141421,0.194179)
\psline[linecolor=colour0, linewidth=0.04](1.2141422,-0.605821)(0.014142151,-1.805821)(1.2141422,-1.005821)(1.6141422,-0.605821)
\psline[linecolor=colour0, linewidth=0.04, linestyle=dashed, dash=0.17638889cm 0.10583334cm](1.6141422,-0.605821)(2.4141421,0.194179)
\psline[linecolor=colour0, linewidth=0.04](5.614142,-0.605821)(5.614142,0.994179)(6.814142,1.794179)(6.814142,0.194179)(5.614142,-0.605821)(7.2141423,-0.605821)(8.414143,0.194179)(6.814142,0.194179)
\psline[linecolor=colour0, linewidth=0.04](5.614142,-0.605821)(4.414142,-1.805821)(5.614142,-1.005821)(5.782258,-0.8377051)
\psline[linecolor=colour0, linewidth=0.04, linestyle=dashed, dash=0.17638889cm 0.10583334cm](5.8460264,-0.76813984)(6.814142,0.194179)
\psline[linecolor=colour1, linewidth=0.04](4.414142,0.594179)(6.814142,-1.805821)(8.014142,-1.005821)(7.614142,-0.605821)
\psline[linecolor=colour1, linewidth=0.04, linestyle=dashed, dash=0.17638889cm 0.10583334cm](7.614142,-0.605821)(5.8502812,1.1731585)
\psline[linecolor=colour1, linewidth=0.04](4.408345,0.5999761)(5.614142,1.4057732)(5.8314986,1.1854833)
\end{pspicture}
}}
\caption{a book and a bookcrossing}\label{picbook}
\end{figure}

\begin{definition}[Intersection number via displacement rule]\label{definters}
I. If $\Q^n$ is decomposed into the direct sum of two subspaces $U_1\oplus U_2$, we call a pair of tropical fans of the form $q_1\cdot U_1$ and $q_2\cdot U_2$ a {\it crossing}, and define the intersection number of these fans as $$q_1 q_2 \left|\frac{\Z^n}{(U_1\cap \Z^n)+(U_2\cap \Z^n)}\right |.$$

II. Let $\mathcal{F}$ and $\mathcal{G}\in K(\Z^n,\Q^n)$ be tropical fans of complementary dimension. Then, for almost every $u\in \Q^n$, 
the shifted support sets $\supp\mathcal{F}+u$ and $\supp\mathcal{G}$ intersect transversally, i.e. the intersection is a finite set $\{v_1,\ldots,v_N\}$, and the localizations $\mathcal{F}_{v_i-u}$ and $\mathcal{G}_{v_i}$ for every $i=1,\ldots,N$ form a crossing $C_i$ (or, informally speaking, the ``shifted fans'' $\mathcal{F}+u$ and $\mathcal{G}$ in a small neighborhood of the intersection point $v_i$ coincide with the crossing $C_i$ up to a shift). Then the local intersection number of $\mathcal{F}_{v_i-u}$ and $\mathcal{G}_{v_i}$ at $v_i$ is defined as the intersection number of the crossing $C_i$ (see I), and the total intersection number $$\mathcal{F}\circ\mathcal{G}$$ is defined as the sum of the local ones. Tropicality of $\mathcal{F}$ and $\mathcal{G}$ ensures that the result is independent of the choice of the displacement $u$ (see Figure \ref{picind} for an example).
\end{definition}

\begin{figure}\centering\includegraphics{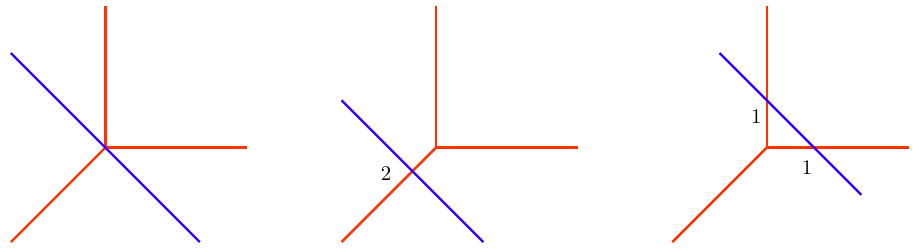}
\caption{two tropical fans with unit weights and the intersection number 2}\label{picind}
\end{figure}

\begin{definition}[Multiplication via intersection numbers]\label{defmultconstr}
I. A pair of tropical fans $\mathcal{F}$ and $\mathcal{G}\in K(\Z^n,\Q^n)$ is said to be a {\it bookcrossing}, if  $\Z^n$ admits a decomposition $L_1\oplus L_2$ such that $\mathcal{F}=\mathcal{F}_1\times\mathcal{F}_2$ and $\mathcal{G}=\mathcal{G}_1\times\mathcal{G}_2$, where $\mathcal{F}_1$ and $\mathcal{G}_1\in K(L_1)$ have complementary dimension, and $\mathcal{F}_2=\mathcal{G}_2\in K(L_2)$ is the whole space $L_2\otimes\Q$ (see Figure \ref{picbook}). Then the product $$\mathcal{F}\cdot\mathcal{G}$$ is defined as the subspace $L_2\otimes\Q\subset \Q^n$, endowed with the constant weight $\mathcal{F}_1\circ\mathcal{G}_1$. 

For example, according to the computation on Figure \ref{picind}, the product of the bookcrossing on Figure \ref{picbook} equals the intersection line equipped with the weight 2.

II. The product of arbitrary fans $\mathcal{F}$ and $\mathcal{G}\in K(\mathcal{C})$ in the ambient cone $\mathcal{C}=(L,C)$ is defined as the fan $\mathcal{H}\in K(\mathcal{C})$ such that, for every $v\in C$,

$\bullet$ if the localizations $\mathcal{F}_v$ and $\mathcal{G}_v$ form a bookcrossing, then $$\mathcal{H}_v=\mathcal{F}_v\cdot\mathcal{G}_v,$$ where the right hand side is defined in (I). In particular,
the weight $w_\mathcal{H}(v)$ equals the intersection product from (I).

$\bullet$ otherwise $v\notin \supp\mathcal{H}$.
\end{definition}

\subsection{Local tropical fans} We have extended the notion of weighted and tropical fans in the space $\Q^n$ to fans in an arbitrary ambient cone $\mathcal{C}=(L,C)$. In brief, this generalization can be summarized as follows.
\begin{definition} \label{defloctrop} 1. Let $Z_k(L)$ be the vector space of $k$-dimensional weighted fans in the lattice $L$. 
Every fan $\mathcal{F}\in Z_k(L)$ uniquely decomposes as $$\mathcal{F}=\mathcal{F}_1+\mathcal{F}_2,$$ where

$\bullet\quad\dim\;\supp(\mathcal{F}_1)\setminus C<k$, i.e., informally speaking, $\mathcal{F}_1$ is the intersection of $\mathcal{F}$ with the cone $C$.

$\bullet\quad\supp(\mathcal{F}_2)\cap C=\varnothing$, i.e., informally speaking, $\mathcal{F}_2$ is the intersection of $\mathcal{F}$ with the complement to $C$.

We denote the corresponding decomposition of $Z_k(L)$ by $$Z_k(\mathcal{C})\oplus O_k(\mathcal{C}).$$
The first summand is referred to as the {\it space of weighted fans in the ambient cone} $\mathcal{C}$.

2. Let $K(L)\subset Z(L)=\bigoplus_{k}Z_k(L)$ be the ring of tropical fans in the lattice $L$ (i.e. weighted fans satisfying the balancing condition). Its image under the projection $Z(L)\to Z(\mathcal{C})$ will be denoted by $K(\mathcal{C})$. The operation of multiplication of tropical fans survives this projection, and the resulting ring $K(\mathcal{C})$ is referred to as the {\it ring of tropical fans in the ambient cone} $\mathcal{C}$.
\end{definition}
\begin{remark} 1. The set of all $k$-dimensional tropical fans in $K(\mathcal{C})$ will be denoted by $K_k(\mathcal{C})$ or $K^{n-k}(\mathcal{C})$, so that we have the graded ring $K(\mathcal{C})=\bigoplus_{m=0}^n K^{m}(\mathcal{C})$.

2. An epimorphism of ambient cones $p(\mathcal{C})=\mathcal{C}'$ gives rise to the linear map $p_*:K_k(\mathcal{C})\to K_k(\mathcal{C}')$, induced by taking the direct image of tropical fans.

3. Let $C_v$ be the localization of the cone $C$ at $v\in L$, i.e. the cone generated by $C$ and $-v$. The operation of localization of tropical fans induces a ring homomorphism $K(\mathcal{C})\to K(\mathcal{C}_v)$, the localization of a tropical fan $\mathcal{F}\in K_k(\mathcal{C})$ is denoted by $\mathcal{F}_v\in K_k(\mathcal{C}_v)$.

4. The ring $K(\Z^n,\Q^n)$ is the conventional ring of tropical fans, serving as a combinatorial model for the ring of conditions of the complex torus $\CC^n$, i.e. the direct limit of the singular cohomology rings of all toric compactifications of the complex torus $\CC^n$, see \cite{dcp}. The ring $K(\Z^n,\Q^n_+)$ similarly reflects the intersection theory of germs of analytic sets in $(\C^n,0)$; more generally, if $C$ is a pointed cone, then the ring $K(\mathcal{C})$ reflects the intersection theory of germs of analytic sets in the affine toric variety of the cone $C$. The precise general statement is given in the next section.
\end{remark}

\section{Local tropical characteristic classes}

\subsection{Tropicalization} We introduce the operation of tropicalization, assigning local tropical fans to germs of analytic sets. 
The results and constructions are generally similar to the more well known global case of algebraic sets (see e.g. \cite{ekl}, \cite{tev1}, \cite{sty}, \cite{mcst}), and the proofs can be found in the literature devoted to the analytic setting (see e.g. \cite{local1}, \cite{local2} and especially \cite{local3}). 
\begin{remark} If a reader motivated by the monodromy conjectures is primarily interested in the polynomial case, it is enough to assume that all analytic germs below are given by polynomials, adopting the original polynomial setting. 
\end{remark}

We shall work with toric varieties (see e.g. \cite{toric1}) and their fans, which we sometimes call toric fans to avoid confusion with weighted and tropical fans. Recall that a toric fan is a collection of non-overlapping relatively open rational polyhedral cones, closed with respect to taking faces.
\begin{definition}\label{defcfan} A toric fan is called a $\mathcal{C}$-fan for an ambient cone $\mathcal{C}=(L,C)$ (Definition \ref{defambient}), if the union of its cones equals the closure of the cone $C$, and is said to be simple, if every its $n$-dimensional cone is generated by a basis of the lattice $L$. \end{definition}

Let $\Sigma$ be a simple $\mathcal{C}$-fan, and let $X_\Sigma$ be the corresponding toric variety. Identify the character lattice of its dense complex torus $T\cong\CC^n$ with the dual to the lattice $L$ of the ambient cone. Let $X^c_\Sigma$ be the union of the precompact orbits of $X_\Sigma$ (i.e. the orbits corresponding to the cones of $\Sigma$ outside the boundary of $C$). We call $X^c_\Sigma$ the compact part of $X_\Sigma$ (and it is indeed compact). 
\begin{example}\label{exatoricpair} Assume that $C=\{a_1<0\}$ in the plane with the standard coordinates $(a_1,a_2)$. Consider the $\mathcal{C}$-fan $\Sigma$ with two 2-dimensional cones equal to the two coordinate quadrants in $C$. The corresponding toric variety $X_\Sigma$ is the product of $\CP^1$ with the standard coordinates $(u:v)$ and $\C^1$ with the standard coordinate $x$. Its compact part $X^c_\Sigma$ is the projective line $\CP^1\times\{0\}$.
\end{example}
\begin{definition} \label{deftropsn} 1) A $\mathcal{C}$-germ of an analytic set is the intersection of $T$ with a germ of an analytic set in the pair $(X_\Sigma, X^c_\Sigma)$ for a $\mathcal{C}$-fan $\Sigma$ (the definition is independent of the choice of $\Sigma$).

2) Assume that the closure of a $\mathcal{C}$-germ $P$ in $X_\Sigma$ of the pure dimension $d$ intersects the precompact (and hence all) orbits $O\subset X_\Sigma$ properly (so that $\codim O\cap \bar P=\codim O+\codim P$). Then the compactification $(X_\Sigma, X^c_\Sigma)$ is said to be tropical for $P$.

3) Let $C_i$ be the (relatively open) interior $d$-dimensional cones in the fan $\Sigma$ of a tropical compactification $(X_\Sigma, X^c_\Sigma)$ of $P$. Let $m_i$ be the intersection number of the orbit corresponding to $C_i$ and the closure of $P$ (it is correctly defined, because the intersection is 0-dimensional by the definition of a tropical compactification). Then the linear combination of weighted fans $\sum_i m_i\cdot C_i$ is a tropical fan in $K_d(\mathcal{C})$, called the tropicalization of $P$ and denoted by $\Trop P$. 
\end{definition}

\subsection{Newton polyhedra} We shall illustrate this definition with some examples in the important special case of a hypersurface, when the germ $P$ is given by one analytic equation $$\sum\nolimits_{a\in(\mbox{\scriptsize a finite subset of }\Z^n)+C^*} c_ax^a=0,\eqno{(*)}$$ where  $x^a$ stands for the monomial $x_1^{a_1}\ldots x_n^{a_n}$. 
\begin{definition}\label{defnewt} 1) A germ of an analytic function given by the series $(*)$ converging for all $x\in T$ in a neighborhood of the compact part of a toric variety $X^c_\Sigma$ for a $\mathcal{C}$-fan $\Sigma$ is called a $\mathcal{C}$-germ of an analytic function. Its Newton polyhedron is defined as the convex hull of $C^*+\{a\,|\, c_a\ne 0\}$.

2) A $\mathcal{C}$-fan is said to be compatible with the Newton polyhedron $N$, if, for any two vectors $u$ and $v$ from the same (relatively open) cone of the $\mathcal{C}$-fan, their support faces are equal: $N^u=N^v$ (or, equivalently, if the support function of the polyhedron is linear at every cone of the fan and is not defined outside).
\end{definition}
Recall that every lattice polyhedron is compatible with some simple fan (\cite{toroidal}).
\begin{example}\label{exanewt1} 1. If $\mathcal{C}=(\Z^n,\Q^n)$, we obtain the usual notion of the Newton polytope of a polynomial. 

2. If $\mathcal{C}=(\Z^n,\Q_+^n)$, we obtain the usual notion of the Newton polyhedron of a germ of an analytic function $(\C^n,0)\to(\C,0)$. 

3. For any $\mathcal{C}$, the tropicalization of the hypersurface $(*)$ is the (first) dual fan $[N]$ of its Newton polyhedron $N$ (Definition \ref{defdual}). In particular, a compactification $(X_\Sigma, X^c_\Sigma)$ is tropical for the hypersurface $(*)$ if and only if $\Sigma$ is compatible with $N$.

4. For instance, assume $C$ is as on Figure \ref{picdual}, and $(x,y)$ are the standard coordinates on $T=\CC^2$. Then, for any toric $\mathcal{C}$-fan $\Sigma$, the coordinate function $x$ extends to a regular function $x$ on the toric surface $X_\Sigma$, and the compact part $X^c_\Sigma$ is given by the equation $x=0$. 

The polar cone $C^*$ is the first positive coordinate ray in the plane, and, as an example of power series of the form $(*)$, we shall consider $f(x,y)=x^2yf_1(x)+f_2(x)+xy^{-1}f_3(x)+x^2y^{-2}f_4(x)$, where $f_i$ are univariate power series. The assumption that $f$ converges in a neighborhood of $X^c_\Sigma$, i.e. for $|x|\leqslant\varepsilon$, is equivalent to convergence of each univariate $f_i$. 

If $f_i(0)\ne 0$, then the Newton polyhedron of $f$ is as shown on Figure \ref{picdual}, so the tropicalization of the germ of the curve $f(x,y)=0$ for small $x$ equals the dual fan on Figure \ref{picdual}.
\end{example}

\subsection{Algebraic and topological interpretation of tropicalization} For varieties of arbitrary codimension, the tropicalization can be defined in more algebraic terms similarly to hypersurfaces.

\begin{definition}\label{definit} 1) Every vector $v\in C$ defines a valuation on the ring $R_{\mathcal{C}}$ of power series of the form $(*)$. For an element $f\in R_{\mathcal{C}}$, its initial part with respect to this valuation is denoted by $f^v$ and called the $v$-initial part of $f$. More explicitly, if $f$ is of the form $(*)$, and $N$ is its Newton polyhedron, then $$f^v(x)=\sum_{a\in N^v}c_ax^a\in\C[x_1^{\pm 1},\ldots,x_n^{\pm 1}].$$

2) For an ideal $I$ in the ring $R_{\mathcal{C}}$ and a vector $v\in C$, the $v$-initial parts of the elements of $I$ generate an ideal in $\C[x_1^{\pm 1},\ldots,x_n^{\pm 1}]$ denoted by $I^v$ and called the $v$-initial part of $I$.

3) For a $\mathcal{C}$-germ of an analytic set $P$, the $v$-initial part of its defining radical ideal $I_P^v$ defines an algebraic set $P_v\subset T$ that will be called the $v$-initial part of $P$.

4) The tropicalization of a $\mathcal{C}$-germ of a $d$-dimensional analytic set $P$ is a $d$-dimensional tropical fan $\Trop P$, represented by the following pre-fan $F$:

$\bullet$ the support set $S_F$ consists of all $v$ such that $I_P^v$ is the tensor product of $\C[x_1^{\pm 1},\ldots,x_d^{\pm 1}]$ and a non-trivial 
ring $R_v$ of finite length;

$\bullet$ the weight function $w_F$ is defined as $w_F(v)=l(R_v)$ for all $v\in S_F$.
\end{definition}
\begin{example}\label{exanewt2} Continuing Example \ref{exanewt1}.4, the $(-1,-1)$-initial part of $f$ equals $f^{(-1,-1)}=f_2(0)+xy^{-1}f_3(0)+x^2y^{-2}f_4(0)=(z-z_1)(z-z_2)$, where $z=xy^{-1}$. Thus the quotient by the initial ideal $\C[x^{\pm 1},y^{\pm 1}]/I_{f=0}^{(-1,-1)}$ equals $\C[x^{\pm 1}]\otimes\C[z^{\pm 1}]/\langle(z-z_1)(z-z_2)\rangle$, and the second multiplier $R_v$ is a Cohen--Macaulay ring of length $2=w(-1,-1)$.
\end{example}

The trpoicalization can be understood in terms of the ring of conditions \cite{dcp}. Due to the local setting, we prefer singular homology to Chow rings, as in \cite{kaz}. Denote the $j$-th homology group of the pair $(X_\Sigma, X_\Sigma\setminus X^c_\Sigma)$ by $H^c_j(X_\Sigma)$. The relative Poincare duality induces a ring structure on the sum $H^c(X_\Sigma)=\bigoplus_j H^c_j(X_\Sigma)$. The inclusions $H^c(X_\Sigma)\hookrightarrow H^c(X_{\Sigma'})$ induced by the blow-ups $X_{\Sigma'}\to X_{\Sigma}$ for subdivisions $\Sigma'$ of $\Sigma$ define the structure of a directed system on the collection of rings $H^c(X_\Sigma)$ over all toric $\mathcal{C}$-fans $\Sigma$. 
\begin{definition} The $\mathcal{C}$-local ring of conditions of the torus $T$ is the direct limit of the rings $H^c(X_\Sigma)$ with respect to subdivisions of $\Sigma$. The fundamental class of a $\mathcal{C}$-germ of an analytic set $P$ in the ring of conditions is the fundamental class of its closure in the homology of its tropical compactification.
\end{definition}
\begin{proposition} The $\mathcal{C}$-local ring of conditions of the torus $T$ is naturally isomorphic to the ring $K(\mathcal{C})$. This isomorphism is uniquely defined by the condition that, for every $\mathcal{C}$-germ of an analytic set $P$, the fundamental class of $P$ in the ring of conditions is sent to the tropicalization of $P$. 
\end{proposition}

\subsection{Tropical characteristic classes} Every affine $k$-dimensional algebraic variety gives rise to a $k$-dimensional tropical fan called the tropicalization of the variety. It turns out that one can naturally refine this correspondence, assigning to a $k$-dimensional variety a collection of tropical fans of dimensions $0,1,\ldots,k$, of which the highest-dimensional one is the tropicalization. This collection (or rather its formal sum) is called the tropical characteristic class and keeps much more geometric information about the algebraic variety. It is convenient to introduce tropical characteristic classes for analytic sets with multiplicities in the following sence. 
\begin{definition}\label{defmacph} 1) A linear combination of the characteristic functions of finitely many analytic $\mathcal{C}$-germs with rational coefficients is called a $\mathcal{C}$-germ of a constructible function. The vector space of such functions will be denoted by $A(\mathcal{C})$.

2) For an epimorphism of tori $\pi:T\to T'$ and the corresponding linear map of their Lie algebras $p$, such that $p^{-1}(\mathcal{C}')=\mathcal{C}$, the MacPherson direct image \cite{mcph} is the unique linear map $$\pi_*:A(\mathcal{C})\to A(\mathcal{C}')$$ that sends the characteristic function of every $\mathcal{C}$-germ $M\subset T$ to $\varphi:T'\to\Q$ such that $\varphi(y)$ is the Euler characteristic $e\bigl(M\cap \pi^{-1}(y)\bigr)$.
\end{definition}
In what follows, we identify a germ of an analytic set with its characteristic function.
\begin{definition} \label{deftcc} The (local) tropical characteristic class is the (unique) linear map $$\langle\cdot\rangle:A(\mathcal{C})\to K(\mathcal{C})$$ that sends every $\mathcal{C}$-germ $P$ to an element $$\langle P\rangle=\sum_k \langle P\rangle_k,\; \langle P\rangle_k\in K^k(\mathcal{C}),$$ and satisfies the following properties:

1 (normalizing):  For any $\mathcal{C}$-germ of a codimension $k$ analytic set $P$, its characteristic class equals
$$\langle P\rangle_k+\ldots+\langle P\rangle_n,$$
where $\langle P\rangle_k$ is the tropicalization $\Trop P$, and, in the global case $C=L\otimes\Q$, the Euler characteristic $e(P)$ equals $\langle P\rangle_n\in K_0(\mathcal{C})=\Q$ (note that in the local case $C\subsetneq L\otimes\Q$ we have $\langle P\rangle_n\in K_0(\mathcal{C})=0$).

2 (intersection): For any $\mathcal{C}$-germs of analytic sets $P$ and $Q$, we have $$\langle gP\cap Q \rangle=\langle P \rangle\cdot \langle Q \rangle$$ for almost all elements $g\in T$, where $gP=\{gz\,|\, z\in P\}$ is the set $P\subset T$ shifted by the element $g$ of the group $T$.

3 (product): For any $\mathcal{C}$-germ $P$ and $\mathcal{C}'$-germ $Q$, we have $$\langle P\times Q \rangle=\langle P \rangle \times \langle Q \rangle.$$

4 (pushforward): For an epimorphism of tori $\pi:T\to T'$ and the corresponding linear map of their Lie algebras $p$, such that $p^{-1}(\mathcal{C}')=\mathcal{C}$, every constructible function $\varphi\in A(\mathcal{C})$ satisfies $$\langle \pi_*\varphi \rangle=p_* \langle \varphi \rangle,$$
where $\pi_*$ is the MacPherson direct image, and $p_*$ is the direct image of tropical fans (Definition \ref{defprojfan}).

5 (adjunction): If, for some vector $v\in C$, the initial part $P_v$ of a germ of an analytic set $P$ is regular (scheme-theoretically, in the sense of the defining ideal $I_P^v$), then $$\langle P\rangle_v=\langle P_v\rangle,$$ i.e. the localization of the characteristic class of $P$ at $v$ equals the characteristic class of the $v$-initial part of $P$.

6 (CSM): Consider a constructible function $\varphi:T\to\Q$ and a simple $\mathcal{C}$-fan $\Sigma$. If the support set of the tropical fan $\langle \varphi\rangle_k\in K^k(\mathcal{C})$ is contained in the codimension $k$ skeleton of $\Sigma$ for all $k=0,\ldots,n$, i.e. the characteristic class $\langle \varphi\rangle$ is contained in the subring $H^c(X_\Sigma)\subset ($ring of conditions$)\cong K(\mathcal{C})$, then this class equals the Chern-Schwartz-MacPherson class (\cite{schw}, \cite{mcph}) of the constructible function $\varphi:X_\Sigma\to\Q$ extended to the complement $X_\Sigma\setminus T$ by $0$.
\end{definition}

The existence of local tropical characteristic classes having these properties can be proved with the same explicit construction as in the global case  
(see \cite{E17}), see the next section.
We shall be especially interested in the following special case.
 
\begin{definition} A compactification $(X_\Sigma, X^c_\Sigma)$ is said to be smooth for a $\mathcal{C}$-germ of an analytic set $P$, if the toric variety $X_\Sigma$ is smooth (i.e. the toric fan $\Sigma$ is simple) and the closure of $P$ in $X_\Sigma$ is smooth and transversal to every orbit of $X^c_\Sigma$ (and hence of $X_\Sigma$).
\end{definition}
In this case, properties 1, 3 and 5 of the tropical characteristic class imply the following.
\begin{corollary}\label{schon} Let $\Sigma$ be a simple $\mathcal{C}$-fan such that the compactification $(X_\Sigma, X^c_\Sigma)$ is smooth for a $\mathcal{C}$-germ $P$. Then the characteristic class $\langle P\rangle_d$ equals the sum of the weighted fans $\sum_i e_i \cdot C_i$, where $C_i$ are the codimension $d$ interior cones in $\Sigma$, and $e_i$ is the Euler characteristic of the intersection $\bar P\cap(C_i$-orbit of $X^c_\Sigma)$.
\end{corollary}
\begin{remark} 1. The existence of characteristic classes for algebraic sets that admit a smooth compactification is independently proved in \cite{gross}. This proof as well 
extends to analytic germs. 

2. A similar construction for matroids in \cite{shaw} (overlapping with ours in the case of linear subvarieties of the complex torus) may be interesting in the context of monodromy conjecture because of the special role of $B$-faces in the non-degenerate case (see \cite{L-V}, \cite{ELT}).
\end{remark}

 \begin{example}
\label{tccnondeg} 
1. A hypersurface $f=0$ of the form $(*)$ admits a smooth compactification, if it is non-degenerate, i.e. 0 is a regular value of the polynomial $f^v:T\to\C$ for every vector $v$ in the interior of $C$. In this case, a smooth compactification is given by any simple fan compatible with the Newton polyhedron of $f$.

2. For a polyhedron $N$ denote the sum $\sum_d (-1)^{d-1}[N]^d$ by $\langle N\rangle$, then, for every non-degenerate hypersurface $f=0$ with the Newton polyhedron $N$ we have $\langle f=0\rangle=\langle N\rangle$.

3. In particular, continuing Example \ref{exanewt2}, the germ $f(x,y)$ is non-degenerate if $f_3^2(0)-4f_2(0)f_4(0)\ne 0$. The tropical characteristic class of the curve germ $f=0$ is just the dual fan $[N]$ shown on Figure \ref{picdual}. 

4. Similarly, for an algebraic curve $C$ in the complex torus $\CC^2$ given by a polynomial equation with the Newton polygon $N$, the tropical characteristic class equals $[N]+e(C)$ (where $e$ is the Euler characteristic). If the curve is non-degenerate, then the Kouchnirenko--Bernstein--Khovansii formula \cite{khovci} ensures that $e(C)=-[N]^2$, and the tropical characteristic class equals $\langle N\rangle$. For degenerate curves this equality does not hold true. In particular, the tropical characteristic class of a germ is not defined by the tropicalization of this germ; it is a strictly stronger invariant.
\end{example}

\section{Relative tropical fans and a construction for tropical characteristic classes}

The aim of this section is to provide a construction for tropical characteristic classes, in order to prove that they exist (we do not need this construction in what follows). The construction is essentially the same as for the global case in \cite{E17}, as soon as we use the right target ring, somewhat different from $K(\mathcal{C})$. 
Let us introduce it.

\subsection{Relative tropical fans} 
\begin{definition}\label{defreltrop}
I. For an $n$-dimensional ambient cone $\mathcal{C}=(L,C)$, the space of $k$-dimensional {\it relative tropical fans} $KK_k(\mathcal{C})=KK^{n-k}(\mathcal{C})$ is the space of pairs $$\{(P,Q)\,|\, P \mbox{ and } Q\in K_k(L),\, \mbox{ such that the support of } P-Q \mbox{ is in } \bar C\}$$ 
(with the componentwise operations) modulo its subspace $$\{(R,R)\,|\, \mbox{ the support of }R\mbox{ avoids }C\}.$$

II. Sums, direct images, localizations and products of relative tropical fans are defined componentwise. The corresponding ring structure on the direct sum of vector spaces $KK(\mathcal{C})=\bigoplus_{k=0}^n KK^{k}(\mathcal{C})$ is called the ring of relative fans.
\end{definition}
\begin{example} Sending an element $(p,q)\in KK_0(\mathcal{C})$ to $p-q$, we identify $KK_0(\mathcal{C})$ with $\Q$. In particular, the product of relative fans $\mathcal{F}\in KK^k(\mathcal{C})$ and $\mathcal{G}\in KK^{n-k}(\mathcal{C})$ is identified with the number that will be called the intersection number of $\mathcal{F}$ and $\mathcal{G}$ and denoted by $\mathcal{F}\circ\mathcal{G}$.
\end{example}
This definition may not seem too natural at the first glance, but its geometrical meaning is clarified by its connection to a relative version of the mixed volume, see below. Its crucial advantage over the ring of local fans $K_k(\mathcal{C})$ is that $KK_0(\mathcal{C})$ equals $\Q$ for all ambient cones $\mathcal{C}$, while $K_0(\mathcal{C})=0$ unless $\mathcal{C}=(\Q^n,\Z^n)$. This difference will be employed in the key Theorem \ref{thlocexist}. There is also a non-trivial point in Definition \ref{defreltrop}.
\begin{lemma}
The product of relative fans $(P_i,Q_i)\in KK^{k_i}(\mathcal{C}),\, i=1,2$, does not depend on the choice of representatives $(P_i,Q_i)$.
\end{lemma}
\begin{proof} By Definition \ref{defmultconstr}, the localization reduces the statement to the case $k_1+k_2=n$. In this case, it is enough to prove that, if $(P_1,Q_1)=(R,R)$, and the support of the fan $R$ avoids $C$, then the intersection number vanishes, i.e. $R\circ P_2=R\circ Q_2$. By Definition \ref{definters}, in order to compute the intersection number of $P_i$'s (or $Q_i$'s), we should shift $P_i$'s and $Q_i$'s by generic displacement vectors $u_i\in L\otimes\Q$ and count local contributions of the intersection points $v_1,\ldots,v_M$ of shifted $P_i$'s (or intersection points $w_1,\ldots,w_N$ of shifted $Q_i$'s). However, if we choose $u_1=0$ and generic $u_2\in C$, then, for every point $v$ of the support set of $R$, the shifted fans $P_i$ and $Q_i$ will locally coincide: $$(P_i)_{v-u_i}=(Q_i)_{v-u_i},$$
because $v-u_i\notin\bar C$. Thus, we have $M=N$ and $v_j=w_j$ for all $j$, and the contribution of each of these intersection points to the intersection number of $P_i$'s is the same as for $Q_i$'s.
\end{proof}
\begin{remark}\label{embrel} 1) The ring $K(\mathcal{C})$ embeds into $KK(\mathcal{C})$, sending a fan $\mathcal{F}$ to the pair $(\mathcal{\tilde F},\mathcal{\tilde F})$, where $\mathcal{\tilde F}$ is an arbitrary global fan in $K(L)$, whose restriction to the cone $C$ equals $\mathcal{F}$. We shall denote the relative fan $(\mathcal{\tilde F},\mathcal{\tilde F})$ by the same letter $\mathcal{F}$ when it comes to adding and multiplying it with other relative tropical fans, and whenever the result is independent of the choice of $\mathcal{\tilde F}$.

2) Moreover, the intersection number of germs of analytic sets can be reconstructed from the intersection of their tropical fans in the ring $KK$. However, for this purpose we should send a germ of an analytic sent to a more sophisticated element of $KK$ than the one described in the preceding remark.

Namely, if a germ of an analytic set $X$ in $\C^n$ intersects all positive-dimensional coordinate subtori properly, then its local tropical fan $P$ in the ring $K(\Z^n,\R^n_+)$ can be uniquely extended to the global tropical fan $\tilde P\in K(\Z^n)$ whose restriction to the octant $\R^n_+$ equals $P$, so that the pair $(0,\tilde P)$ is 
a relative tropical fan.
Denoting this relative tropical fan by ${\rm RTrop} X\in K(\Z^n,\R^n_+)$, the intersection number of the germs $X$ and $Y$ of complementary dimension in $\C^n$ equals the intersection number of their relative tropical fans ${\rm RTrop} X\cdot {\rm RTrop} Y \in K_0(\Z^n,\R^n_+)=\Q$. 
(We do not need this fact for the purpose of the present paper.)
\end{remark}

\subsection{Relative mixed volumes} 
Let $\mathcal{C}=(L,C)$ be an $n$-dimensional lattice cone and $C^*\subset L^*\otimes\Q$ its polar cone.
\begin{definition}[\cite{E05},\cite{Edet},\cite{KhT}] I. A relative polyhedron compatible with $\mathcal{C}$ is a pair of closed polyhedra $(A,B)$ such that

$\bullet$ each of $A$ and $B$ is of the form $($bounded lattice polytope in $L^*)+C^*$, where ``$+$'' is the Minkowski summation;

$\bullet$ the symmetric difference of $A$ and $B$ is bounded.

The set $PP(\mathcal{C})$ of all relative polyhedra compatible with $\mathcal{C}$ is a semigroup with respect to the componentwise Minkowski summation.

II. The dual fan $[(A,B)]\in KK^1(\mathcal{C})$ of the relative polyhedron $(A,B)$ is the relative fan with a representative of the form $([\tilde A],[\tilde B])$, where $\tilde A\subset A$ and $\tilde B\subset B$ are bounded lattice polytopes such that $\tilde A\setminus\tilde B=A\setminus B$ and $\tilde B\setminus\tilde A=B\setminus A$.

III. The volume of a relative polyhedron $\Vol(A,B)$ is the difference of the lattice volumes of  $\tilde A\setminus\tilde B$ and $\tilde B\setminus\tilde A$.

IV. The mixed volume of relative polyhedra is the unique symmetric multilinear function 
$$\MV:\underbrace{PP(\mathcal{C})\times\ldots\times PP(\mathcal{C})}_{n}\to\Z$$
such that $\MV\left( (A,B),\ldots,(A,B) \right)=\Vol (A,B)$.
\end{definition}
\begin{figure}
\centering\includegraphics{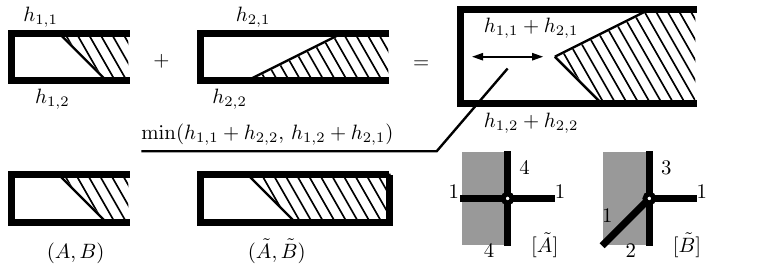}
\caption{Minkowski sums and dual fans of relative polyhedra}\label{picrelvol}
\end{figure}
\begin{lemma} 1. The dual fan does not depend on the choice of $\tilde A$ and $\tilde B$.

2. Taking the dual fan is a linear map $PP(\mathcal{C})\to KK^1(\mathcal{C})$.

3. We have $\MV\left( (A_1,B_1),\ldots,(A_n,B_n) \right)=[(A_1,B_1)]\cdot\ldots\cdot[(A_n,B_n)]$.
\end{lemma}
Part 1 is obvious, 2 and 3 follow from the same properties of non-relative polytopes.
\begin{example} Examples of the Minkowski sum and the dual fan of relative polyhedra are given on Figure \ref{picrelvol} (the ambient cone is gray, the first polyhedron $A$ in a relative polyhedron $(A,B)$ is bold, the second polyhedron $B$ is hatched).
\end{example}

\subsection{Relative Kouchnirenko--Bernstein--Khovanskii formula}
Let $\Sigma$ be a simple $\mathcal{C}$-fan for an ambient cone $\mathcal{C}=(L,C)$ (Definition \ref{defcfan}). Its one-dimensional cones are generated by primitive vectors $v_i\in L$ and correspond to codimension 1 orbits of the corresponding toric variety $X_\Sigma$. Denote the closures of these orbits by $D_i$.

Let $A$ be a polyhedron, for which $\Sigma$ is the minimal compatible fan (i.e. the $n$-dimensional cones of $\Sigma$ are exactly the linearity domains of the support function of $A$), and denote the value of the support function $A(v_i)$ by $m_i$. Then there exists a unique very ample line bundle $I_A$ on $X_\Sigma$ equipped with a meromorphic section $s_A$, such that the divisor of zeros and poles of $s_A$ equals $\sum_i m_i D_i$. Moreover, this correspondence, assigning the pair $(I_A,s_A)$ to the
polyhedron $A$, is an isomorphism between the semigroups $\{$convex lattice polyhedra for which $\Sigma$ is the minimal compatible fan$\}$ and $\{$pairs $(I,s)$, where $I$ is a very ample line bundle on $X_\Sigma$, and $s$ is its meromorphic section with zeros
and poles outside the dense torus$\}$.

For every germ $s$ of a holomorphic section of $I_A$ on $(X_\Sigma,X_\Sigma^c)$, the quotient $s/s_A$ is an analytic function on (an open subset of) the dense torus $T$, i.e. a $\mathcal{C}$-germ in the sense of Definition \ref{defnewt}.1. The Newton polyhedron and the $v$-initial part $s^v$ of the section $s$ are defined as the respective objects for $s/s_A$ (Definition \ref{defnewt}.1 and \ref{definit}.1).  
\begin{example}\label{exatoricpair2}
1. If $A$ is the positive orthant in $\Q^n$, and $\mathcal{C}$ is its polar cone, then $(X_\Sigma,X_\Sigma^c)=(\C^n,0)$, and the Newton polyhedron and initial part of a germ $s:(\C^n,0)\to(\C,0)$ have their usual meaning.

2. Continuing Example \ref{exatoricpair}, the polyhedron $A=[0,1]\times[0,+\infty)$ (bold on Figure \ref{picrelvol}) gives rise to the line bundle $I_A$ on $X_\Sigma=\CP^1\times\C^1$ which is a pull-back of $\mathcal{O}(1)$ on $\CP^1$. A typical section of this bundle has the form $s=ux^{h_1}f_1(x)+vx^{h_2}f_2(x)$ for some univariate power series $f_1$ and $f_2$ with non-zero constant terms. 
The Newton polytope of such section is hatched on Figure \ref{picrelvol}.
\end{example}
\begin{definition} \label{defmilnor} Let $A_1,\ldots,A_k$ be lattice polyhedra compatible with a simple $\mathcal{C}$-fan $\Sigma$,  
and let $s_i$ be germs of holomorphic sections of the line bundles $I_{A_i}$ on $(X_\Sigma,X_\Sigma^c)$. Let $\varepsilon_i$ be generic sections (of the same bundles) small enough relatively to a small enough tubular neighborhood $U$ of the compact part $X_\Sigma^c$.
The {\it Milnor fiber} of the tuple $(s_1,\ldots,s_k)$ on a germ of an analytic set $S$ in $(X_\Sigma,X_\Sigma^c)$ is the set 
$$\{s_1=\varepsilon_1\}\cap\ldots\cap\{s_k=\varepsilon_k\}\cap S\cap U\cap T.$$
\end{definition}
\begin{theorem}[Relative Kouchnirenko--Bernstein--Khovanskii formula, \cite{Edet}, \cite{E11}]\label{relkbk} I. In the setting of Definition \ref{defmilnor}, assume that the sections $s_i$ of the line bundles $I_{A_i}$ have Newton polyhedra $B_i\subset A_i$, such that $A_i\setminus B_i$ is bounded, and generic initial parts in the sense that $s_1^v=\ldots=s_k^v=0$ defines a regular algebraic subvariety of the torus $T$ for every vector $v\in C$. 

Then the Euler characteristic of the Milnor fiber of $(s_1,\ldots,s_k)$ on the toric variety $X_\Sigma$
equals $$(-1)^{n-k}\sum_{n_1>0,\ldots,n_k>0\atop n_1+\ldots+n_k=n} (A_1,B_1)^{n_1}\cdot\ldots\cdot(A_k,B_k)^{n_k},$$
where the product of $n$ relative polyhedra stands for their relative mixed volume.

II. In particular, if $k=n$, then the intersection of the sets $\{s_1=\varepsilon_1\},\ldots,\{s_k=\varepsilon_k\}$ is contained in the compact part $X_\Sigma^c$, so the intersection index of these sets is correctly defined and equals the relative mixed volume of $(A_1,B_1),\ldots,(A_n,B_n)$. 
\end{theorem} 
\begin{example} 1. If $A$ is the positive orthant in $\Q^n$, and $\mathcal{C}$ is its polar cone, then $(X_\Sigma,X_\Sigma^c)=(\C^n,0)$, and we obtain the well known 
formula  (\cite{kouchn} for $k=1$ and \cite{oka} in general)  
for the Milnor number of a non-degenerate complete intersection germ in $(\C^n,0)$.

2. Continuing Example \ref{exatoricpair2}.2, consider two sections $s_i=ux^{h_{i,1}}f_{i,1}(x)+vx^{h_{i,2}}f_{i,2}(x)$ of the line bundle $I_A$ and assume $f_{i,j}(0)\ne 0$. Then the Newton polyhedra of the sections $s_i$ are hatched on Figure \ref{picrelvol}. In particular, the Euler characteristic of the Milnor number of $s_1$ equals the volume of the first relative polyhedron on Figure \ref{picrelvol}, i.e. $h_{1,1}+h_{1,2}$. This agrees with the fact that the coordinate map $x:X_\Sigma\to\C$ identifies the Milnor fiber of $s_1$ with a disc with $h_{1,1}+h_{1,2}+1$ punctures. 

The Milnor fiber of $(s_1,s_2)$ is a finite set whose cardinality equals the intersection index of the divisors $s_1=0$ and $s_2=0$ at the compact component $\CP^1\times\{0\}$ of their intersection. This number equals the mixed volume of the two relative polyhedra on Figure \ref{picrelvol}, i.e. $\min(h_{1,1}+h_{2,2},\,h_{1,2}+h_{2,1})$. This agrees with the fact that the determinant of the matrix 
$\begin{bmatrix} x^{h_{1,1}}f_{1,1}(x) & x^{h_{1,2}}f_{1,2}(x)\\
x^{h_{2,1}}f_{2,1}(x) & x^{h_{2,2}}f_{2,2}(x)
\end{bmatrix}$ as a function of $x$ has a root of order $\min(h_{1,1}+h_{2,2},\,h_{1,2}+h_{2,1})$ at 0.
\end{example}

\subsection{A construction for local tropical characteristic classes} 
\begin{definition} Let $\mathcal{C}$ be an $n$-dimensional ambient cone. A function $\Phi:PP(\mathcal{C})\to\R$ is said to be a {\it polynomial starting from} $(A,B)\in PP(\mathcal{C})$ if there exist local fans $\mathcal{F}_i\in K_i(\mathcal{C})$ such that
$$\Phi(A',B')=\mathcal{F}_n\cdot[(A',B')]^n+\ldots+\mathcal{F}_0\cdot[(A',B')]^0\in KK_0(\mathcal{C})=\Q$$
for every $(A',B')\in PP(\mathcal{C})$ whose support functions satisfy the condition
$$A'(\cdot)-B'(\cdot)>A(\cdot)-B(\cdot).$$
\end{definition}
The multiplication in this definition is understood in the sense of Remark \ref{embrel}.
\begin{theorem} \label{thlocexist} In the setting of Definition \ref{defmilnor}, assume that the sections $s_i$ of the line bundles $I_{A_i}$ have Newton polyhedra $B_i\subset A_i$, such that $A_i\setminus B_i$ is bounded, and generic principal parts.

I. The Euler characteristic of the Milnor fiber of $s_1$ on the germ of an analytic set $S$ is a polynomial 
$$\mathcal{F}_n\cdot[(A',B')]^n+\ldots+\mathcal{F}_0\cdot[(A',B')]^0$$
starting from some relative polyhedron.

II. The coefficients $\mathcal{F}_i\in K^{n-i}(\mathcal{C})$ satisfy the properties of the tropical characteristic classes $\langle S\rangle_{n-i}\in K^{n-i}(\mathcal{C})$ of the set $S$ (Definition \ref{deftcc}). This in particular proves the existence of tropical characteristic classes.

III. The Euler characteristic of the Milnor fiber of $(s_1,\ldots,s_k)$ on $S$ equals the 0-dimensional component of the product
$$\langle S\rangle\cdot\langle (A_1,B_1)\rangle\cdot\ldots\cdot\langle (A_k,B_k)\rangle,$$
where $\langle (A,B)\rangle=\sum_{i>0} (-1)^{i-1} [(A,B)]^i$, and the value is taken in $KK_0(\mathcal{C})=\Q$.
\end{theorem}
The proof for the global case $\mathcal{C}=(\Q^n,\Z^n)$ is given in \cite{E17}, Sections 2.4-2.6. The proof in the general case is the same verbatim, adding ``relative'' to all tropical fans, polytopes and references to the Kouchnirenko--Bernstein--Khovanskii formula.

\section{Tropical nearby monodromy eigenvalues}

We come back to the assumptions and notation of Definition \ref{defemb} and choose a complex number $s=\exp{2\pi i k/m}$ for coprime $k$ and $m$, assuming that the denominator $m$ is not tautological in the following sense.
\begin{definition} The number $m$ is called a {\it tautological denominator} for $f:(\C^n,0)\to(\C,0)$, if $m$ divides one of $m_j$'s in the prime factorization $f=\prod_j f_j^{m_j}$.
\end{definition}

The aim in this section is to introduce tropical nearby monodromy eigenvalues of a singularity with a given embedded toric resolution, in the sense of Definition \ref{defemb}. The material is difficult to illustrate immediately, because the first substantial examples come from non-obvious singularities of four variables (see e. g. Example 7.4 in \cite{ELT}). However, t. n. e. of small codimension will be studied in a much more explicit fashion in the next section.

We summarize the setting of Definition \ref{defemb} in the following diagram:
$$\begin{tikzcd}
 & Y \arrow[r,symbol=\subset] \arrow{d}{} & W \arrow[r,symbol=\subset] \arrow{d}{} & Z \arrow{d}{\pi} \\
0  \arrow[r,symbol=\in] & X \arrow[r,symbol=\subset] & S \arrow[r,symbol=\subset] & \C^N \\
0  \arrow[r,symbol=\in, pos=0.75] \arrow[u,symbol=\mapsto] & \{f=0\} \arrow[r,symbol=\subset, pos=0.25] \arrow{u}{} & \mathcal{S} \arrow[swap]{r}{f} \arrow{u}{j} &\C
\end{tikzcd}
$$
Here $Z\to\C^N$ is a toric blow-up, $\mathcal{E}$ is the set of its orbits, and $Y\subset W$ are the strict transforms of $X\subset S$, intersecting every orbit $E\in\mathcal{E}$ by the smooth sets $Y_E=Y\cap E$ and $W_E=W\cap E$ respectively. The set of orbits in $\pi^{-1}(0)$ is denoted by $\mathcal{E}_0\subset\mathcal{E}$.

\subsection{A stratification of $W$}
In order to study monodromy of $f$ in terms of its embedded toric resolution, it is important to detect points of the total space $W$ at which the function $f\circ i^{-1}\circ\pi:W\to\C$ equals a monomial $x^M$ for some local coordinate function $x:W\to\C$. Such points form so called $V$-strata, as the subsequent lemma shows.
 
\begin{definition} I. A stratum of the embedded toric resolution $(j,\pi)$ is a connected component of the set $W_E\setminus Y$ for some orbit $E\in\mathcal{E}$. 

 II. A stratum $C$ is said to be {\it convenient}, if we have $\pi(C)\not\subset X$ (i.e. $f$ does not vanish on $j^{-1}\circ\pi(C)$). 

III. A non-convenient stratum $B\subset Z$ is called a {\it $V$-stratum}, if it is in the closure of a convenient stratum $C$  
such that $\dim C=\dim B+1$.
\end{definition}
\begin{remark} 1. A $V$-stratum is contained in a unique convenient stratum of the dimension greater by 1.

2. If $f$ is non-degenerate with respect to its Newton polyhedron $N$, then the convenient strata correspond to faces of $N$ that belong to coordinate planes of the same dimension, and the $V$-strata correspond to the $V$-faces of $N$ (including the unbounded ones, see Definition \ref{defvface} below).
\end{remark}

\begin{lemma}\label{lacampo} Let $E$ be an orbit of $Z$, and $a\in W_E$.  Then, in a small neighborhood of $a$, the function $f\circ i^{-1}\circ\pi:W\to\C$ can be represented as $x^{M_E(a)},$ where $M_E(a)$ is a positive integer number and $x$ is a local coordinate function on $W$, 
if and only if

1) the number $M_E(a)$ is a tautological denominator, or

2) the point $a$ belongs to a $V$-stratum. 
\end{lemma}
\begin{proof} In a small neighborhood of $a$, choose a coordinate system $y_1,\ldots,y_n$ on $W$, such that 

-- $W_E$ is the coordinate plane $y_1=\ldots=y_k=0$, 

-- the adjacent sets of the form $W_{E'}$ are parameterized by the subsets $I\subset\{1,\ldots,k\}$ and are given by the conditions $y_i=0$ for $y\in I$ and $y_i\ne 0$ for $i\notin I$, and 

-- $Y$ is empty or given by the equation $y_n=0$.  

There are two options for the coordinate function $x$ in the statement of the lemma: it can either be chosen as $y_{i_0}$ for some $i_0\leqslant k$, or as $y_n$. In the first case, we have $a\notin Y_E$, i.e. $a$ belongs to some stratum, and, moreover, this is a $V$-stratum (because it is in the closure of the convenient stratum given by the conditions $y_{i_0}\ne 0$ and $y_i=0$ for other $i\leqslant k$).
In the second case, $W_E\setminus Y_E$ is itself a convenient stratum near $a$, but $a\in Y_E$, thus $M_E(a)$ is a tautological denominator.
\end{proof}
Given an orbit $H$ of the toric resolution $Z$, its closure $\bar H$ is a toric variety, corresponding to a certain $\mathcal{C}_H$-fan $\Sigma_H$. The ambient cone $\mathcal{C}_H$ and the toric fan $\Sigma_H$ can be explicitly described in terms of the toric fan $\Sigma$ of the toric variety $Z$ and its ambient cone $\mathcal{C}$ (which is either the positive orthant in $\Q^N$ or opposite to it, depending on the convention in the beginning of Section \ref{schoice1}).

Namely, let $P$ be the cone of $\Sigma$, corresponding to the orbit $H$. Then the ambient cone $\mathcal{C}_H$ is the image of $\mathcal{C}$ under the quotient map $p:\Q^N\to\Q^N/($vector span of $P)$, and the toric fan $\Sigma_H$ consists of the images of the cones of $\Sigma$ adjacent to $P$ under the projection $p$.

\begin{definition} For an integer number $M>0$ and an orbit $E\subset Z$  
in the closure of an orbit $H\subset Z$, define the $\mathcal{C}_E$-germ of the semi-analytic set
$$\Phi_{M,H,E}=\left(\overline{\{a\,|\, M_H(a)=M\}}\cap E\right)\setminus Y.$$
\end{definition}
\begin{remark}\label{remphi}
1. In other words, the set $\Phi_{M,H,H}$ is the subset $W_H\setminus Y$ at which the locally constant function $M_H(\cdot)$ equals $M$, and $\Phi_{M,H,E}$ is the intersection of $E\setminus Y$ with the closure of $\Phi_{M,H,H}$.

2. By this definition, every set $\Phi_{M,H,E}$ is a union of strata. 
\end{remark}

The tropical characteristic class of $\Phi_{M,H,H}$ is the principal ingredient in the definition of tropical nearby monodromy eigenvalues, so it is important to compute them explicitly.
\begin{theorem}\label{thexplic} The characteristic class $\langle\Phi_{M,H,H}\rangle_d\in K^d(\mathcal{C}_H)$ can be expressed in terms of the Euler characteristics of strata of the embedded toric resolution as follows: it equals the sum of weighted fans $$\sum_E e(\Phi_{M,H,E}) P_E,$$ where $e$ is the Euler characteristic, $E$ ranges through all codimension $d$ toric orbits in the compact part of the toric variety $\bar H$, and $P_E$ are the corresponding codimension $d$ cones in the toric fan $\Sigma_H$.
\end{theorem}
\begin{proof} By Remark \ref{remphi}.1, the formula is a special case of Corollary \ref{schon}. By Remark \ref{remphi}.2, the Euler characteristic of $\Phi_{M,H,E}$ is a sum of Euler characteristics of strata.
\end{proof}

\subsection{Tropical nearby eigenvalues} For $I\in\{1,\ldots,N\}$, let $\C^I$ and $\Q^I$ be the subsets of $\C^N$ and $\Q^N$ respectively, defined by vanishing of $j$-th coordinates for all $j\notin I$. We can naturally consider $\C^I$ as an affine toric variety, corresponding to the cone $C^I\subset\Q^I$, and let $\CC^I$ be the big torus in $\C^I$. 
\begin{remark} As always, the cone $C^I$ is either the positive orthant in $\Q^I$ or the opposite one, depending on the convention in the beginning of Section \ref{schoice1}.
\end{remark}
The cone $C^I$ comes with the lattice dual to the character lattice of $\CC^I$, and together they form an ambient cone that we denote by $\mathcal{C}^I$. Let $\mathcal{E}_I$ be the set of all orbits of $Z$ that project to $\CC^I$.

\begin{definition}\label{deftne} We say that a complex number $s=\exp{2\pi i k/m}$ for coprime $k$ and $m$ with a non-tautological denominator $m$ is a tropical nearby monodromy eigenvalue in $\C^I$, if we have
$$\Phi_{m,I}=\sum_{H\in\mathcal{E}_I,\, m|M} \pi_*\langle\Phi_{M,H,H}\rangle\ne 0$$ 
in $K(\mathcal{C}^I)$. It is said to be tropical in codimension $d$, if the codimension $d\leqslant |I|$ component of $\sum_{H\in\mathcal{V}_I,\, m|M} \pi_*\langle\Phi_{M,H,H}\rangle$ is non-zero in $K^{d}(\mathcal{C}^I)$.
\end{definition}
\begin{remark} 1. This terminology implies that the same nearby monodromy eigenvalue may be tropical in several different codimensions.

2. Tropical nearby monodromy eigenvalues by definition depend only on the adjacency and the Euler characteristics of precompact strata of the embedded toric resolution.
\end{remark}
\begin{theorem}\label{main} 
Every tropical nearby monodromy eigenvalue is indeed a nearby monodromy eigenvalue.
\end{theorem}
For the proof, define the constructible function $F_{M,I}$ as the sum of the MacPherson direct images $\sum_{H\in\mathcal{V}_I} \pi_*\Phi_{M,H,H}$. Applying A'Campo's formula \cite{AC} for the Milnor monodromy $\zeta$-function $\zeta_{x_0}$ of the germ of $f$ at an arbitrary point $x_0\in\CC^I$, we obtain the following by Lemma \ref{lacampo}.
\begin{theorem}\label{acampo} We have $\zeta_{x_0}(t)=\prod_M (1-t^M)^{F_{M,I}(x_0)}$.
\end{theorem}
We can now prove Theorem \ref{main}.
\begin{proof} 
If $\Phi_{M,I}=\langle F_{M,I}\rangle\ne 0$, then there exists a point $x_0\in\CC^I$ such that $\sum_{m|M} F_{M,I}(x_0)\ne 0$. Thus, by Theorem \ref{acampo}, the denominator $m$ gives rise to a root or pole of $\zeta_{x_0}$.
\end{proof}

\subsection{Specializing to Newton polyhedra}\label{sspecnewton} Let us specialize to the case of a classical toric resolution $\pi:Z\to\C^n$ of a non-degenerate singularity $f:(\C^n,0)\to(\C,0)$ with respect to its Newton polyhedron $N$, i.e. assume that $i:(S,0)=(\C^n,0)\hookrightarrow(\C^N,0)$ is the identity map.

\begin{definition} \label{defvface} A (not necessarily bounded) face $\Gamma\subset N$ is called a {\it $V$-face}, if it is contained in a coordinate subspace $\Q^J$ such that $|J|=\dim\,\Gamma+1$. More precisely, it is called a $V_I$-face, where $I\subset J$ is a unique minimal (by inclusion) subset such that the projection of $\Gamma$ along $\R^I$ is bounded. 
\end{definition}
\begin{remark}\label{remvi} The image of $\Gamma$ under the projection along $\R^I$ is a bounded $V$-face of the image $N^I$ of the polytope $N$. Conversely, every $V_I$-face is the preimage of a bounded $V$-face under the projection $N\to N^I$.
\end{remark}
\begin{example} A $V_I$-face $\Gamma$ and its image $A$ under the projection along $\R^I$ are shown in bold on Figure \ref{picvi} for $I=\{1\}$ and $J=\{1,2\}$.
\end{example}
\begin{figure}
\centering\includegraphics{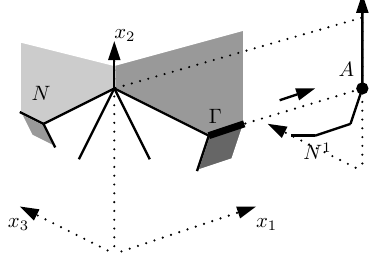}{}
\caption{A $V_I$-face $\Gamma$ for $I=\{1\}$ is shown in bold}\label{picvi}
\end{figure}

For a $V_I$-face $\Gamma\subset\Q^J$, let $m_\Gamma$ be its lattice distance to 0 in $\Q^J$, i.e. $|\Z^J/U_\Gamma|$, where $U_\Gamma$ is the vector subspace generated by pairwise differences of the points of $\Gamma$.  There is a natural projection $\pi$ from the dual space $U_\Gamma^*$ to the dual space of $\Q^I$. We denote $C_\Gamma=\pi^{-1}(C^I)$, and this cone together with the lattice dual to $U_\Gamma\cap\Z^n$ forms an ambient cone $\mathcal{C}_\Gamma$. 

Recall that, for a polyhedron $P$ in $\Q^p$, we denote the formal sum $\sum_p (-1)^{p-1}[P]^p$ by $\langle P\rangle$, see Definition \ref{defdual} for the dual tropical fan $[P]^p$.  
For a $V$-face $\Gamma$, consider the element $\langle \Gamma\rangle\in K(\mathcal{C}_\Gamma)=\bigoplus_k K^k(\mathcal{C}_\Gamma)$.
\begin{definition}\label{deftropicalpoly}
We say that a complex number $s=\exp{2\pi i k/m}$ for coprime $k$ and $m$ with a non-tautological denominator $m$ is a tropical nearby monodromy eigenvalue of the polyhedron $N$ in $\C^I$, if we have
$$\Phi_{m,I}=\sum_{\Gamma\subset N \mbox{ \scriptsize is a } V_I-\mbox{\scriptsize face,}\atop m|m_\Gamma} \pi_*\langle \Gamma\rangle\ne 0.$$ 
It is said to be tropical in codimension $d$, if the codimension $d$ component of $\sum_{\Gamma\subset N \mbox{ \scriptsize is a } V_I-\mbox{\scriptsize face}\atop m|m_\Gamma} \pi_*\langle \Gamma\rangle$ is non-zero in $K^{d}(C^I)$.
\end{definition}
According to the Example \ref{tccnondeg}, the sum $\langle \Gamma\rangle$ equals the tropical characteristic class of a generic hypersurface with the Newton polyhedron $\Gamma$. This fact makes Definition \ref{deftropicalpoly} the special case of Definition \ref{deftne}, and Theorem \ref{main} implies the following.
\begin{corollary}\label{tropicalpoly} If $s$ is a tropical nearby monodromy eigenvalue of the polyhedron $N$ in $\C^I$, then, for every non-degenerate singularity $f:(\C^n,0)\to(\C,0)$ with the Newton polyhedron $N$, there exists $x_0\in\CC^I$ arbitrary close to 0 such that $s$ is a monodromy eigenvalue of the germ of $f$ at $x_0$.
\end{corollary}

\section{Low codimension}

In this section we make the notion of tropical nearby monodromy eigenvalues in codimension 0 and 1 more explicit. This simplification is based on the notion of a fiber polyhedron.

\subsection{Fiber polytopes} Denote $\{1,\ldots,n\}\setminus I$ by $\bar I$. The coordinate projection of a polyhedron $N\subset\Q^n$ along $\Q^{I}$ and its image will be denoted by $\pi_I:N\to N^I\subset\Q^{\bar I}$.

\begin{definition}[\cite{bs}] For every polytope $B\subset N^I$, the fiber polyhedron $\int_B N$ of $N$ over $B$ is defined (up to a shift) as the set of all points of the form $(1+\dim B)\int_B s(x)\; dx$, where $s$ runs over all continuous sections $B\to N$ of the projection $\pi_I:N\to N^I$, and $dx$ is the lattice volume form on $B$.
\end{definition}
\begin{remark} 1. The fiber polyhedron is contained in an affine plane parallel to $\Q^I$, so it is convenient to consider instead its isomorphic image under the projection $\Q^n\to\Q^I$ along $\Q^{\bar I}$. 

2. The multiplier $(1+\dim B)$ is introduced to ensure that for every lattice polytope $B$ the fiber polyhedron is also a lattice polyhedron.
\end{remark}
\begin{example} If $\bar I$ consists of one element, and $N$ is bounded, then $\int_{N^I} N$ in the line $\Q^{\bar I}$ is a segment, and its length equals the lattice volume of $N$. Figure \ref{picfiber} gives an example of a two-dimensional fiber polyhedron with $I=\{1,2\}$ and $n=3$.
\end{example}
\begin{figure}\centering\includegraphics{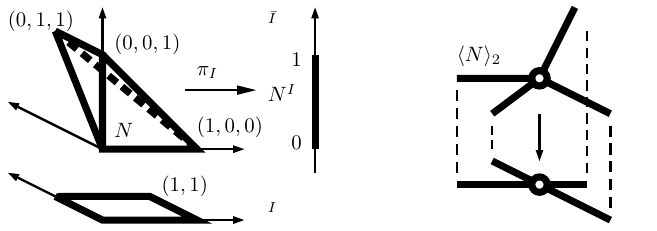}
\caption{the fiber polygon of a three-dimensional polytope for $I=\{1,2\}$}\label{picfiber}
\end{figure}

We come back to the setting and notation of Definition \ref{deftropicalpoly}, assuming that $B=\pi_I(\Gamma)$.
\begin{proposition}\label{tccfib} 1) The codimension 0 component of the projection $(\pi_*\langle\Gamma\rangle)_0$ is the cone $C^I$ with the weight equal to the lattice volume of $B$.

2) The codimension 1 component of the projection $(\pi_*\langle\Gamma\rangle)_1$ is the (first) dual fan of the fiber polyhedron $\int_B\Gamma=\int_B N$.
\end{proposition}
The first part is well known, and the second one is proved in \cite{sty} for the case of bounded polytopes. The proof in the general case is the same. Note that the second statement completely characterizes the fiber polytope $\int_B N$ by Remark \ref{remembed}.
\begin{example} 1. Figure \ref{picdual} illustrates the first statement: the exterior normal rays to the bounded edges of a given polygon $N$ form its first dual fan, and the image of this fan under the vertical projection equals a horizontal ray equipped with the weight 3. This weight equals the lattice height of the polygon. 

2. Figure \ref{picfiber} illustrates the second part of the statement: the exterior normal rays to the facets of a given polytope $N$ form its second dual fan, and the image of this fan under the vertical projection equals the first normal fan of the fiber polygon $\int_{N^I} N$.
\end{example}
Denote $$\Psi_{m,I}=\sum_B (-1)^{\dim B}\int_B N,$$
where $B$ runs over all bounded $V$-faces of $N^I$, such that $m|m_B$.
\begin{remark} The set of polyhedra of the form (bounded polytope$)+C^I$ is a semigroup with respect to Minkowski summation. We subtract polyhedra as elements of the Grothendieck group of this semigroup. 
\end{remark}
\begin{definition} The aforementioned Grothendieck group will be referred to as the group of virtual polyhedra. The neutral element of this group is the cone $C^I$, so a virtual polyhedron is said to be trivial, if it equals $C^I$ up to a shift.
\end{definition}
 Note that if a virtual polyhedron is known to be a polyhedron, then its non-triviality is equivalent to having a bounded edge.
\begin{corollary}\label{coroltne} In the setting of Corollary \ref{tropicalpoly}:

1) the number $s$ is a tropical nearby monodromy eigenvalue of the polyhedron $N$ in codimension 0 in $\C^{I}$, if it is a zero or pole of the Varchenko function
$$\zeta_{x_0}=\prod_B (1-t^{m_B})^{(-1)^{\dim B}\Vol B}$$
where $B$ ranges through all bounded $V$-faces of $N^I$, and $\Vol B$ is the lattice volume. 

2) the number $s$ is a tropical nearby monodromy eigenvalue of the polyhedron $N$ in codimension 1 in $\C^{I}$, if the (virtual) polyhedron $\Psi_{m,I}$ is non-trivial.
\end{corollary}
This follows from Corollary \ref{tropicalpoly} and Proposition \ref{tccfib}.

\subsection{Fiber polyhedra over segments} Motivated by the preceding corollary, we start to study the virtual polyhedron $\Psi_{m,I}$. We assume w.l.o.g. that $I=\{1,\ldots,p\}$ and write the tuple $x=(x_1,\ldots,x_n)$ as $(y,z)$, where $y$ stands for the first $p$ variables, and $z$ for the rest. In particular, we shall write the germ $f(x)$ of an analytic function non-degenerate with respect to its Newton polyhedron $N$ as $f(y,z)=\sum_b f_b(y) z^b$ and denote the restriction $$\sum_{b\in B} f_b(y) z^b\eqno{(\star)}$$ by $f^B$ for a face $B\subset N^I$.

If $B$ is a segment, we can analyse its contribution to $\Psi_{m,I}$ by using its relation to the Newton polyhedra of discriminants. In this setting, the classical discriminant $D_B$ is the equation of the closure of the set of all $y_0\in\CC^{I}$, such that the hypersurface $f^B(y_0,\cdot)=0$ is not regular (i.e. 0 is a critical value of the polynomial function $f^B(y_0,\cdot):\CC^{\bar I}\to\C$). Explicitly, $D_B$ is the Sylvester discriminant of the univariate polynomial $f_{b_k} t^k+\ldots+f_{b_1} t+f_{b_0}$, where $b_0,b_1,\ldots,b_k$ are the consecutive lattice points of the segment $B$. For any $b\in N^I$, denote $\pi_I^{-1}(b)\subset N$ by $N_b$.
\begin{example} For $n=3$ and $I=\{1,2\}$, consider $f(y_1,y_2,z)=f_0+f_1z+f_2z^2$, where $f_0(y_1,y_2)=y_1^2+y_1y_2^2+\ldots$, $f_1(y_1,y_2)=y_2+\ldots$, and $f_2(y_1,y_2)=y_1^3+y_2+\ldots$ (dots stand for the terms above the Newton diagram). The Newton polyhedron $N$ of the germ $f$, its fiber polytope over $B=N^I$, the fibers $N_{b_i}$, and the hatched Newton polyhedron $\Delta_B$ of the discriminant $D_B(f)=f_1^2-4f_0f_2=y_2^2-4y_1y_2-4y_1^5+\ldots$ are shown on Figure \ref{picfibsegm}. This illustrates the following expression for $\Delta_B$.
\end{example}
\begin{figure}\centering\includegraphics{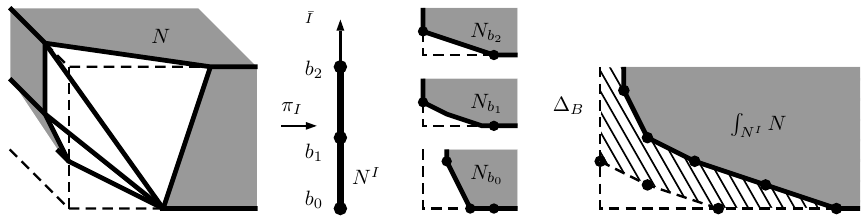}
\caption{a fiber polyhedron over a segment and the Newton polyhedron of the discriminant}\label{picfibsegm}
\end{figure}

\begin{proposition}[\cite{pgp}]\label{positive}  
If the germ $f$ is non-degenerate, then the Newton polyhedron of the discriminant $D_B(f)$ equals $$\Delta_B=\int_B N-N_{b_0}-N_{b_k}.$$
In particular, this difference is a polyhedron.
\end{proposition}
See \cite{E11} for the generalization of this fact to $B$ of arbitrary dimension.
In what follows, it will be important to know when the polyhedron $\Delta_B$ is trivial, i.e. equals the cone $C^I$ up to a shift. 
\begin{example}\label{exatrivdiscr} If $k=1$, i.e. $B$ is a primitive segment, then $\Delta_B$ is trivial, since the discriminant of a polynomial of degree 1 is trivial.
\end{example}
Hence, for an arbitrary $k$, splitting $B$ into primitive segments $B_1,\ldots,B_k$, we observe that $$\Delta_B=\sum_{j=1}^k\Delta_{B_j}+2\sum_{i=1}^{k-1} N_{b_i}=2\sum_{i=1}^{k-1} N_{b_i}.$$
Here every $\Delta_{B_j}$ vanishes, because, upon a suitable dilatation (making $N_{b_j}$ and $N_{b_{j+1}}$ lattice polyhedra), it falls into the scope of Example \ref{exatrivdiscr}. This equality implies the following.
\begin{lemma}\label{trivial}
The polyhedron $\Delta_B$ is trivial if and only if $N_{b_i}$ is trivial for every $i=1,\ldots,k-1$.
\end{lemma}

\subsection{T. n. e. in codimension one} Given a Newton polyhedron $N$, we explicitly describe all its t.~n.~e. of a small codimension at points of the coordinate plane $\CC^I$ for small $|I|$ or small $|\bar I|$. In particular, this gives a necessary and sufficient supply of monodromy eigenvalues to prove the Denef--Loeser monodromy conjecture for non-degenerate functions of four vairables, see \cite{ELT}.
\begin{example}\label{exasmallcodim}
1. For any $I$, the complex number $s=\exp{2\pi i k/m}$ for coprime $k$ and $m$ is a t. n. e. in codimension 0 at $\CC^I$, iff it is a root or pole of the Varchenko formula for the polyhedron $N^I$, see Corollary \ref{coroltne}.1.

2. For $|I|=0$ and $|I|=1$, every t. n. e. at $\CC^I$ has codimension $0$, so this case a special subcase of the preceding one. 

3. There are no t. n. e. for $|\bar I|=0$. For $|\bar I|=1$, the complex number $s=\exp{2\pi i k/m}$ for coprime $k$ and $m$ is a t. n. e. in codimension 0 at $\CC^I$, iff $m|d$, where $d$ is the starting point of the ray $N^I\subset\Q^1$. Note that in this case $m$ is a tautological denominator, because $f$ has a factor $x_i^d,\, i\in \bar I$. The number $s$ is a t. n. e. in codimension 1 at $\CC^I$, only if it is a t. n. e. in codimension 0, and the polyhedron $N_d$ is non-trivial (i.e. has a bounded edge).
\end{example}
We now study t.~n.~e. of codimension one at the coordinate plane $\CC^I,\, |\bar I|=2$, which is the simplest case outside the scope of the Varchenko formula and trivialities from the preceding example.
\begin{theorem}\label{thfinal} I. Assume that $|\bar I|=2$. Then the virtual polyhedron $\Psi_{m,I}$ is a polyhedron.

II. Moreover, it is non-trivial, i.e. $s=\exp{2\pi i k/m}$ for coprime $k$ and $m$ is a t.~n.~e. of the polyhedron $N$ in codimension 1 at the coordinate plane $\CC^{I}$, if and only if there exists a point $a\in\Z^I$ satisfying all of the following:

1) $a$ is contained in an edge $B\subset N^I$ such that $m|m_B$;

2) $a$ itself is not a $V$-vertex, or a $V$-vertex such that $m\nmid m_a$;

3) the polyhedron $\pi_I^{-1}(a)$ is non-trivial, i.e. has a bounded edge.
\end{theorem}
\begin{proof} For any face $B$ of $N^I$ define $i_B$ to be 1 if $B$ is a $V$-face such that $m|m_B$, and 0 otherwise. For every edge $B$ of $N^I$ with end points $P$ and $Q$ define the polyhedron $\Psi_B$ to be $i_B\int_BN-i_P N_P-i_Q N_Q$ (this is indeed a polyhedron by Proposition \ref{positive}). The sought polyhedron $\Psi_{m,I}$ is the sum of the polyhedra $\Psi_B$ over all edges $B\subset N^I$, and, by Lemma \ref{trivial}, every summand is trivial unless it contains a point $a$ satisfying the conditions (1-3) above.
\end{proof}
Combining this with Example \ref{exasmallcodim}.1, we obtain the following. We shall say that a number $s$ is contributed by a face $B\subset N^I$, if $B$ is a $V$-face, and $m|m_B$.
\begin{corollary} \label{coroldim4} Assume that $|\bar I|=2$. Then $s=\exp{2\pi i k/m}$ for coprime $k$ and $m$ is a t.~n.~e. of the polyhedron $N$ in codimension 0 or 1 at the coordinate plane $\CC^{I}$, unless one of the following four cases takes place:

1) No $V$-face $B\subset N^I$ contributes $s$.

2) The polyhedron $N^I$ has one edge $B$ with one interior lattice point $p$ and two $V$-vertices $a$ and $b$, all of them contribute $s$, and the polyhedron $N_p$ 
is trivial.

3) $s$ is contributed only by one $V$-vertex $a_i\in N^I$ and its adjacent edge $B$ of lattice length 1, and for the other end point $b\in B,\, b\ne a$, the polyhedron $N_b$
is trivial.

4) The same as (3) for two $V$-vertices and their adjacent edges, provided that these two edges do not coincide.
\end{corollary}

\subsection{Generalizing to higher dimensions} 
Theorem \ref{thfinal} and Corollary \ref{coroldim4} allow to prove the monodromy conjecture for non-degenerate analytic functions of four variables. A key problem in extending this approach to higher dimensions is to generalize Theorem \ref{thfinal} and Corollary \ref{coroldim4} to coordinate planes of arbitrary dimension $|I|$ and, further, to arbitrary codimension of t. n. e.

\begin{problem} \label{probldef} 1) For a given polyhedron $N\subset \Q^n$ of the form $($bounded lattice polytope$)+\Q^n_+$, a subset $I\subset\{1,\ldots,n\}$ and a number $s=\exp{2\pi i k/m}$, prove that the virtual polyhedron $\Psi_{m,I}$ is a polyhedron.

2) Classify all cases when $\Psi_{m,I}$ is trivial, i.e. $s$ is not a codimension 1 t.~n.~e. of the Newton polyhedron $N$ at the coordinate plane $\CC^I$.

3) Furthermore, classify all cases when the fan $\Phi_{m,I}$ (Definition \ref{deftne}) is trivial, i. e. $s$ is not a t. n. e. in any codimension.
\end{problem}
As we see from the proof of Theorem \ref{thfinal}, an important step would be to generalize Lemma \ref{trivial}, i.e. to classify all Newton polyhedra $N$ and faces $B\subset N^I$ such that the following Minkowski linear combination of the fiber polytopes
$$\Delta_B=\sum_{\Gamma\; \mbox{\scriptsize is a face of } B} (-1)^{\codim\Gamma}\int_{\Gamma} N$$
is a non-trivial polyhedron.

It is easy to see that $\Delta_B$ is a non-virtual polyhedron, because it equals the Newton polyhedron of $D_B(f^B(y,\cdot))$, where 

-- $f^B$ is a polynomial of $z$, whose coefficients depend analytically on $y$, see $(\star)$, 

-- $D_B$ is the regular $A$-discriminant of this polynomial of $z$ in the sense of \cite{gkz}, 11.1.A. 

One can moreover prove that $D_B$ itself is trivial if and only if $B$ is a unit simplex, however the triviality of the polyhedron $\Delta_B$ (which, besides $B$, depends on $N$) is a more subtle question.

\noindent
\textsc{
HSE University \\
Faculty of Mathematics NRU HSE, 6 Usacheva, 119048, Moscow, Russia}

\end{document}